\documentclass[dvips,12pt,twoside]{article}
\usepackage{authblk}

\usepackage{ifthen}
\usepackage{graphicx}
\usepackage{epsfig}
\usepackage{nicefrac}
\usepackage{mathrsfs}
\usepackage{amsmath,amsfonts,euscript,amssymb,amsthm}
\usepackage{amsrefs}
\usepackage[left=30mm,right=28mm,top=26mm,bottom=25mm]{geometry}
\usepackage[dvipdfm]{hyperref}

\newcommand{\R}{\mathbb{R}}
\newcommand{\N}{\mathbb{N}}
\newcommand{\C}{\mathbb{C}}

\newcommand{\eps}{\varepsilon}

\newcommand{\ra}{\rangle}
\newcommand{\la}{\langle}

\newcommand{\del}{\partial}

\newcommand{\supp}{\mathrm{supp}}

\newcommand{\one}{\mathrm{\bf 1}}

\newcommand{\loc}{\mathrm{loc}}
\newcommand{\odd}{\mathrm{odd}}
\newcommand{\even}{\mathrm{even}}
\newcommand{\low}{\mathrm{low}}
\newcommand{\high}{\mathrm{high}}
\newcommand{\out}{\mathrm{out}}
\newcommand{\inn}{\mathrm{in}}
\def\vectdue#1#2{\left(\begin{matrix}#1\cr #2\end{matrix}\right)}

\def\calH{\mathcal{H}}

\def\HH{\dot{H}^1(\R^2)}

\newtheorem{theorem}{Theorem}[section]
\newtheorem{definition}[theorem]{Definition}
\newtheorem{remark}[theorem]{Remark}

\newtheorem{lemma}[theorem]{Lemma}
\newtheorem{proposition}[theorem]{Proposition}

\numberwithin{equation}{section}

\title{A variational perspective on cloaking by anomalous localized
  resonance}

\author[1]{R.V. Kohn}
\author[1,2]{J. Lu}
\author[3]{B. Schweizer}
\author[4]{M.I. Weinstein}
\affil[1]{Courant Institute of Mathematical
      Sciences, New York University, 251 Mercer Street, New York, NY
      10012, U.S.A.}
\affil[2]{Mathematics Department,
      Duke University, Box 90320, Durham, NC 27707}
\affil[3]{Technische Universit\"at Dortmund,
      Fakult\"at f\"ur Mathematik, Vogelpothsweg 87, D-44227 Dortmund,
      Germany}
\affil[4]{Department of Applied Physics
      and Applied Mathematics, Columbia University, New York, NY 10027}

\date{\small\today}

\begin{document}

\pagestyle{myheadings} \markboth{A variational perspective on cloaking
   by anomalous localized resonance}{R.V.\,Kohn, J.\,Lu, B.\,Schweizer,
   M.I.\,Weinstein}

\maketitle

\begin{abstract}
  A body of literature has developed concerning ``cloaking by
  anomalous localized resonance''. The mathematical heart of the
  matter involves the behavior of a divergence-form elliptic equation
  in the plane, $\nabla\cdot (a(x)\nabla u(x)) = f(x)$. The
  complex-valued coefficient has a matrix-shell-core geometry, with
  real part equal to $1$ in the matrix and the core, and -1 in the
  shell; one is interested in understanding the resonant behavior of
  the solution as the imaginary part of $a(x)$ decreases to zero (so
  that ellipticity is lost). Most analytical work in this area has
  relied on separation of variables, and has therefore been restricted
  to radial geometries. We introduce a new approach based on a pair of
  dual variational principles, and apply it to some non-radial
  examples. In our examples, as in the radial setting, the spatial
  location of the source $f$ plays a crucial
  role in determining whether or not resonance occurs. 
  
  \medskip
  {\bf MSC:} 35Q60, 35P05

  \medskip
  {\bf Keywords:} cloaking, anomalous localized resonance, negative
  index metamaterials
\end{abstract}

\section{Introduction}

A body of literature has developed concerning ``cloaking by anomalous
localized resonance''.  Cloaking of two types of objects has been
considered: (a) {\it dipoles} or {\it inclusions}, considered e.g. in
\cites{BrunoLintner, Nicorovici-etal1994, Milton-etal2005,
  Milton-etal2006, BouchitteSchweizer2010} and (b) spatially localized
{\it sources} \cite{Ammari-etal}.  In this article we work in the
setting of localized sources.

\begin{figure}[th]
   \centering
   \includegraphics[height=59mm]{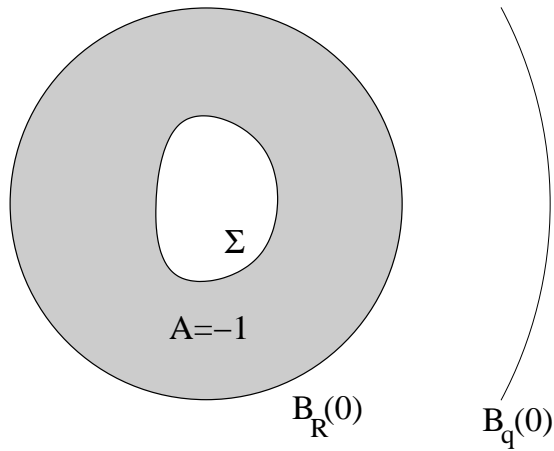}
   \caption{\em Sketch of the core-shell-matrix geometry. 
     \label{fig:ringgeom}}
\end{figure}

Focusing initially on the math (not the physics), we are interested in
a divergence-form PDE in the plane:
\begin{align}
  \label{eq:original}
  \nabla\cdot (a_\eta\nabla u_\eta) &= f\quad \text{ on } \R^2 \\
  \ \  \nabla{u}_\eta &\to 0\ \ {\rm as}\ |x|\to\infty.
  \nonumber
\end{align}
The coefficient $a_\eta(x)$ is piecewise constant and complex-valued,
with constant imaginary part $\eta > 0$:
\begin{equation}
  a_\eta (x) = A(x) + i \eta ; 
  \label{eq:a_eta}
\end{equation}
its real part has a matrix-shell-core character in the sense that 
\begin{equation}
  A(x) = \left\{ 
    \begin{array}{cl} +1 & \mbox{outside $B_R(0)$}\\
      -1 & \mbox{in the shell $B_R(0) \setminus \Sigma$}\\
      +1 & \mbox{in the core $\Sigma$}
    \end{array} \right.  
  \label{eq:A-def}
\end{equation}
(see Figure \ref{fig:ringgeom}). Concerning the core, we assume that
\begin{equation}\label{eq:structure-core}
  \Sigma \subset B_1(0)
\end{equation}
so the shell includes an annulus of width $R-1$.  Concerning the
source $f$, we assume it is real-valued, supported at distance $q$
from the origin, and has zero mean:
\begin{equation}\label{eq:structure-f-at-q}
  f = F\, \calH^1\lfloor \del B_q(0),\quad 
  F:\del B_q(0)\to \R,\ F\in L^2(\del B_q(0)), 
  \text{ and } \int_{\partial B_q(0)} F\, d \calH^1 = 0\,.
\end{equation}
Our interest lies in the question: \medskip

\noindent\textbf{Question:} \textit{As $\eta \rightarrow 0$ with $f$ and
  $A(x)$ held fixed, what is the behavior of}
\begin{equation*}
  E_\eta := \frac{\eta}{2} \int_{\R^2} |\nabla u_\eta|^2 \, dx \ ?   
\end{equation*}
\smallskip

\noindent In the radial setting, i.e. when $\Sigma = B_1(0)$, one
expects by analogy with \cites{Ammari-etal, BouchitteSchweizer2010,
  Milton-etal2006, Milton-etal2005, Nicorovici-etal1994} that the
answer depends mainly on the location of the source. Specifically:
there is a critical radius $R^* = R^{3/2}$ such that for a broad class
of sources $f$,
\begin{align*}
  \limsup_{\eta \rightarrow 0} E_\eta = \infty &  \mbox{ if $q < R^*$, while}\\ 
  \limsup_{\eta \rightarrow 0} E_\eta < \infty &  \mbox{ if $q > R^*$}. 
\end{align*}
Note that it is no restriction to fix the core radius to be $1$. A
scaling argument implies that, for core radius $r_0$, and shell radius
$R$, the critical radius is
\begin{equation*}
  R^*\, =\, r_0 \left(\frac{R}{r_0}\right)^{3/2} =
  \left(\frac{R}{r_0}\right)^{1/2}R. 
\end{equation*}

\begin{definition}[Resonance]\label {def:resonance} 
  Let a configuration be given by coefficients $A$ and source $f$ as
  in \eqref {eq:A-def}--\eqref {eq:structure-f-at-q}.  We shall call
  the configuration \underline{resonant}, if
  \begin{equation*}
    \limsup_{\eta\to 0} E_\eta\ =\ \infty\,.
  \end{equation*}
  Otherwise we call the configuration \underline{non-resonant}.
\end{definition}

In the physics literature the term ``anomalous localized resonance''
is used. An {\it anomalous} feature of the resonance is that it is not
associated to a finite dimensional eigenvalue of a linear operator and
a forcing term at or near the resonant frequency.  Instead, the
resonance here is associated to an infinite dimensional kernel of the
limiting (non-elliptic) operator. The word {\it localized} refers to
the fact that the resonance is spatially localized: while $\int
|\nabla u_\eta |^2 \rightarrow \infty$ if $q < R^*$, the potential
$u_\eta$ (and therefore also its gradient $\nabla u_\eta $) stay
uniformly bounded outside some ball.

The connection to {\it cloaking} is as follows (see \cite{Ammari-etal}
for a more thorough discussion). For time-harmonic wave propagation in
the quasistatic regime, $E_\eta$ is the rate at which energy is
dissipated to heat.  Let us now consider a source $\alpha_\eta f$,
where $\alpha_\eta \in \R$ is a scaling factor. If the (unscaled)
source $f$ produces resonance (i.e. if $E_\eta \rightarrow \infty$)
then the source $\alpha_\eta f$ is connected to the energy dissipation
$\alpha_\eta^2 E_\eta$. If the physical source $\alpha_\eta f$ has
finite power, then we must have $\alpha_\eta \rightarrow 0$ as
$\eta\to 0$. If the fields $u_\eta$ associated with the unscaled
source $f$ are bounded outside a certain region, then the physical
fields $\alpha_\eta u_\eta$ vanish in that region as $\eta\to 0$. This
implies that the finite power source $\alpha_\eta f$ is not visible
from outside.

Returning to a more mathematical perspective: we are interested in
this problem because it involves the behavior of the elliptic system
\eqref{eq:original} in a limit when ellipticity is lost. It is not
surprising that oscillatory behavior occurs in such a limit. It is
however surprising that, at least in the radial examples, (a)
resonance depends so strongly on the location of the source, and (b)
the oscillatory behavior is spatially localized. We would like to
understand the following question: \medskip

\noindent \textbf{Question:} \textit{Is this surprising behavior
  particular to the radial setting, or is it a more general
  phenomenon}?  \smallskip
  
The present paper addresses only point (a): the dependence on the
location of the source.  Our method, which is variational in
character, is unfortunately not well-suited to the study of point (b).

We know only one numerical study of a similar problem with non-radial
(and non-slab) geometry. The paper \cite{BrunoLintner} by Bruno and
Lintner considered, via numerical simulation, various examples
including an elliptical core in an elliptical shell.  The results were
similar to those of the radial case; in particular, the structure
seemed to cloak a polarizable dipole placed sufficiently near the
shell.

The paper \cite{Ammari-etal} by Ammari et al considers a problem very
similar to ours. The main difference is that both the outer and inner
edges of the shell are not constrained to be radial. (There is also a
minor difference: their PDE has $a_\eta = 1$ in the matrix and core
and $a_\eta = -1 + i\eta$ in the shell, so energy is dissipated {\it
  only} in the shell.)  Using a representation based on single layer
potentials, Ammari et al obtain an expression for a spatially
localized analog of $E_\eta$.  To make use of their expression, one
needs detailed information on the spectral properties of certain
boundary integral operators. This information is difficult to come by
in general and hence, beyond the radial setting, it is unclear how to
use their method to obtain information on resonance and non-resonance
in the limit as $\eta \rightarrow 0$.

Our approach is based on variational principles. The starting point is
a pair of (dual) variational principles for $E_\eta$. One expresses
$E_\eta$ as a minimum; trial functions may be used to provide an upper
bound in order to show that resonance doesn't occur. The dual
principle expresses $E_\eta$ as a maximum; trial functions may be used
to provide a lower bound in order to show that resonance
occurs. Similar variational principles were considered
in~\cites{CherkaevGibiansky, Milton-Seppecher-Bouchitte}.  Our main
results -- all proved using the variational principles -- are the
following:
\begin{enumerate}
\item[(i)] If there is no core then there is {\it always} resonance,
  for any source radius $q > R$ and any nonzero $f$ (see Proposition
  \ref {prop:nocore}).
\item[(ii)] For any core $\Sigma \subset B_1(0)$, there is resonance
  for a broad class of sources $f$, provided the source location is $q
  < R^* := R^{3/2}$ (see Theorem \ref {thm:inner-resonance}).
\item[(iii)] In the radial case (when $\Sigma = B_1(0)$), $R^*
  :=R^{3/2}$ is critical, in the sense that (a) when the source
  location is $q < R^*$ resonance occurs for a broad class of $f$'s,
  and (b) when the location is $q> R^*$ resonance does not occur for
  any $f$ (see Theorem \ref {thm:inner-resonance} and Proposition \ref
  {prop:spher-incl}).
\item[(iv)] In the (weakly) nonradial case when the core is
  $B_\rho(z_0)$ with $|z_0|$ sufficiently near $0$ and $\rho$
  sufficiently near $1$, resonance does not occur if the source
  location $q$ is sufficiently large (see Theorem
  \ref{thm:eccentric-localization}).
\end{enumerate}

Point (iii) is already known, from Section 5 of
\cite{Ammari-etal}. Our variational method is interesting even in this
radial setting: our proof of (iii) is, we think, simpler and more
elementary than the argument of \cite{Ammari-etal}. Unfortunately, our
methods do not seem to provide simple proofs for the localization
effect when there is resonance.

In focusing on \eqref{eq:original}, we have chosen the imaginary part
of $a_\eta$ to be the {\it same} constant constant in the matrix,
shell, and core. This simplifies the formulas, and it seems physically
unobjectionable. But we suppose a similar method could be used when
the imaginary part is different in each region.

We specifically consider sources $f$ that are concentrated on the
curve $\del B_q(0)$. Our method also allows the study of more general
distributions of sources, which can be obtained as a superpositions of
concentrated sources, $f$, at different values of $q$.

We have taken the core to have $A(x)=1$ because this case has
particular interest: in the radial setting, the ``cloaking device is
invisible'' if the core has $A=1$, see \cite
{Nicorovici-etal1994}. However anomalous localized resonance also
occurs when $A$ takes a different (constant) value in the core. It
would be interesting to extend our method to analyze cores with $A
\neq 1$.

Our assumption that $A(x)=-1$ in the shell is essential to the
phenomenon. Indeed, our PDE problem becomes very different if the
ratio across each interface, the {\it plasmonic eigenvalue}, is
different from $-1$. This can be seen from the perspective of the
boundary integral method, where ratios other than $-1$ lead to
boundary integral equations of Fredholm type, see \cites{Ammari-etal,
  Grieser:12}

Our main results are almost exclusively for a circular outer shell
boundary $\del B_R(0)$. This is essential to our method, since we use
the perfect plasmon waves on the outer shell boundary in the
construction of comparison functions. We refer to
Section~\ref{sec.general-res-loc} for a further discussion.  A more
general geometry is only treated in
Proposition~\ref{prop:holomorphic-trafo} with the help of a domain
transformation.  Related techniques are used in \cite{Nguyen:12}.

Plasmonic resonance effects have many potential applications. This is
one of the reasons why the development of {\em negative index
  metamaterials} is another much-studied research area, see e.g.
\cites{BouchitteSchweizer-Maxwell, Milton-etal2007, Pendry2004}.  We
hope that our variational approach will be useful also in the other
contexts.

\paragraph{Notation.} We use polar coordinates and write $x\in \R^2$
as $x = r\, (\cos\theta, \sin\theta)$.  In Section~\ref{sec.eccentric}
we identify $\R^2 \equiv \C$ via $(x_1,x_2) \equiv x_1 + i x_2 =
z$. With this notation, we identify $z = r e^{i\theta}$. The complex
conjugate of $z$ is denoted by $\overline{z}$. 

We denote the sphere with radius $\rho$, centered at $x_0$, as
$B_{\rho}(x_0)$. The measure $\calH^1\lfloor \del\Omega$ is the
$1$-dimensional Hausdorff measure on the curve $\del\Omega$. Unless
otherwise specified, integrals are over all of $\R^2$. Constants $C$
may change from one line to the next.

\section{The primal and dual variational principles}

In the subsequent definitions of energies we always consider the
source $f$ as a given element $f\in H^{-1}(\R^2)$. We will always
consider sources with a compact support (in the sense of
distributions).  Furthermore, we shall assume that the sources $f$
have a vanishing average,
\begin{equation*}
  \int_{\R^2}\, f = 0. 
\end{equation*}
Since $f$ is merely a distribution, it would be more correct to write
$\la f, \one\ra = 0$, where $\one : \R^2\to \R$ is the constant
function, $\one(x) = 1$ for all $x\in \R^2$.  We note that since $f$
has compact support, it can be applied to test-functions that are only
locally of class $H^1$.

We remark that, while the main results of this paper concern $\R^2$,
the primal and dual variational principles generalize to any
dimension.

\subsection{A complex elliptic system and its non-elliptic limit}

Our aim is to study, for a sequence $\eta = \eta_j \to 0$, sequences
$u_\eta$ of solutions to \eqref {eq:original}.  For non-vanishing
dissipation, $\eta\ne 0$, \eqref {eq:original} is an elliptic PDE,
while the system loses ellipticity in the limit $\eta\to0$.

To a solution $u_\eta: \R^2 \to \C$ of the original complex-valued
equation
\begin{equation}
  \label{eq:original-copy}
  \nabla\cdot (a_\eta\nabla u_\eta) = f
\end{equation}
we have associated an energy $E_\eta$ (in physical terms the energy
dissipation in the structure)
\begin{equation}
  \label{eq:complex-energy}
   E_\eta(u_\eta) := \frac{\eta}{2} \int_{\mathbb{R}^2} |\nabla u_\eta|^2\,.
\end{equation}
As noted in the introduction, the phenomenon of cloaking is related to
resonance in the sense of Definition \ref{def:resonance},
\begin{equation}
  E_\eta(u_\eta) \to \infty \label{Eblowsup}
\end{equation}
along a subsequence $\eta\searrow 0$.

We can write the complex scalar equation for $u_\eta: \R^2 \to \C$ as
a system of two real scalar equations. We set
\begin{equation}
  \label{eq:u-eta-real-imag}
  u_\eta = v_\eta + i\, \frac1{\eta}\, w_\eta,\quad 
  \text{ with } v_\eta, w_\eta: \R^2 \to \R.
\end{equation}
For a real-valued source, $f:\R^2\to \R$, the complex equation
$\nabla\cdot (a_\eta\nabla u_\eta) = f$ with $a_\eta = A + i\eta$ is
equivalent to the coupled system of two real equations on $\R^2$,
\begin{align}
  \nabla\cdot (A\nabla v_\eta) - \Delta w_\eta &= f,\label{eq:orig-real}\\
  \nabla\cdot (A\nabla w_\eta) + \eta^2 \Delta v_\eta &= 0.\label{eq:orig-imag}
\end{align}
The energy $E_\eta(u_\eta)$ can be expressed in terms of $v_\eta$ and
$w_\eta$ as
\begin{equation}
  E_\eta(u_\eta) = \frac{\eta}{2} \int |\nabla u_\eta|^2
  = \frac{\eta}{2} \int |\nabla v_\eta|^2 + \frac1{2\eta} \int |\nabla w_\eta|^2.
  \label{EI}
\end{equation}
   
In the following subsections we introduce 
\begin{enumerate}
\item the {\it primal variational problem}, a minimization problem,
  which characterizes the energy $E_\eta(u_\eta)$ as a constrained
  minimum; and
\item the {\it dual variational problem}, a maximization problem,
  which characterizes the energy $E_\eta(u_\eta)$ as a constrained
  maximum.
\end{enumerate}

To provide a functional analytic framework for the study of the
variational problems we introduce the following function space of real
or complex-valued functions,
\begin{equation}
  \label{eq:Hdot}
  \HH := \left\{ U\in L^2_\loc(\R^2)\left| 
      \nabla U\in L^2(\R^2)\right.\right\},\quad
  \| U \|_{\HH}^2 := \int_{\R^2} |\nabla U|^2 + \int_{B_1(0)} |U|^2\,.
\end{equation}

\subsection{The primal variational problem}

For fixed $f\in H^{-1}(\R^2)$ we consider the energy functional
\begin{equation}\label{Idef}
  I_\eta(v,w) := \frac{\eta}{2} \int
  |\nabla v|^2 + \frac1{2\eta} \int |\nabla w|^2\,
\end{equation}
defined for $v,w \in\HH$. The primal variational problem is given by
\begin{equation}
  \label{eq:primal}
  \begin{split}
    &{\rm minimize}\ I_\eta(\tilde v,\tilde w)
    \textrm{ over all pairs}\ (\tilde v,\tilde w)\\
    &\text{which satisfy the PDE constraint } \nabla\cdot (A\nabla
    \tilde v) - \Delta \tilde w = f\,.
  \end{split}
\end{equation}

\begin{lemma}\label{lem:primal}
  Let $f\in H^{-1}(\R^2)$ be a fixed real-valued source with compact
  support and with vanishing average. Then the primal variational
  problem \eqref {eq:primal} is equivalent to the original problem
  \eqref {eq:original-copy} with energy \eqref {eq:complex-energy} in
  the following sense.
  \begin{enumerate}
  \item The infimum
    \begin{equation}
      \inf \left\{\ I_\eta(\tilde v,\tilde w)\ 
        \middle\vert \ (\tilde v,\tilde w)\in\HH\times
        \HH,\ \ \nabla\cdot (A\nabla \tilde v) - \Delta \tilde w =
        f  \right\}
      \label{infimum}\end{equation}
    is attained at a pair $(v_\eta,w_\eta)\in\HH\times\HH$.
  \item The minimizing pair, $(v_\eta,w_\eta)$, is unique up to an
    additive constant. The function $u_\eta :=
    v_\eta+i\eta^{-1}w_\eta$ is the unique (up to an additive
    constant) solution of the original problem \eqref
    {eq:original-copy}.
  \item For the solutions, the energies coincide,
    \begin{equation}
      \label{eq:orig-primal}
      E_\eta(u_\eta) = I_\eta(v_\eta,w_\eta)\,.
    \end{equation}
  \end{enumerate}
\end{lemma}

\noindent\emph{Remark}. The lemma implies
\begin{equation}
  \label{eq:lem-primal-conseq}
  E_\eta(u_\eta) \le I_\eta(\tilde v,\tilde w)
\end{equation}
for {\em every} pair $(\tilde v,\tilde w)$ that satisfies the PDE
constraint of \eqref {eq:primal}. We shall use the inequality
\eqref{eq:lem-primal-conseq} to establish non-resonance.

\begin{proof}
  
  \noindent\emph{Point 1.} Fix a radius $s>0$ such that
  $\supp(f) \subset B_s(0)$, we introduce the function space with
  constraint:
  \begin{equation*}
    X :=
    \left\{ \tilde u  \in  \HH\ \middle\vert\ \int_{B_s(0)} \tilde u = 0
    \right\}.
  \end{equation*}
  Note that $I_\eta$, defined in \eqref{Idef}, is convex on $X\times
  X$. Moreover, the constraint set is non-empty. Indeed, choose
  $\tilde v_*$ smooth and of compact support and defined $\tilde w_*$
  to be the weak solution of $\Delta \tilde
  w_*=-f+\nabla\cdot\left(A\nabla \tilde v_*\right)$. It follows that
  the infimum in \eqref{infimum} is attained on $X\times X$ (see,
  e.g., Chapter 8.2 of \cite{Evans}), {\it i.e.} there exists
  $(v_\eta,w_\eta)\in X\times X$ such that
  \[ I_\eta(v_\eta, w_\eta) \le I_\eta(\tilde v,\tilde w)\ \textrm{for
    all}\ (\tilde v,\tilde w)\in X\times X ,\ {\rm with }\
  \nabla\cdot (A\nabla \tilde v) - \Delta \tilde w = f\,.\] 
  
  \smallskip
  \noindent\emph{Point 2.} We first observe that the 
  constraint \eqref {eq:primal} is identical to \eqref {eq:orig-real}.
  As a minimizer of $I_\eta$, the pair $(v_\eta, w_\eta)$ satisfies
  the Euler-Lagrange equation
  \[ \partial_\tau I_\eta\Big(v_\eta+\tau\tilde v, w_\eta+\tau\tilde
  w\Big)\Big|_{\tau=0}\ =\ 0,\text{ for every } (\tilde v,\tilde w)\in
  X\times X\] satisfying $\nabla\cdot (A\nabla \tilde v) - \Delta
  \tilde w = 0$.  For the energy $I_\eta$, this equation reads
  \begin{align*} 
    0 =  \eta \int \nabla v_\eta\cdot \nabla\tilde v + \frac1{\eta} \int
    \nabla w_\eta\cdot \nabla \tilde w = \eta \int \nabla v_\eta\cdot
    \nabla \tilde v + \frac1{\eta} \int \nabla w_\eta\cdot A \nabla
    \tilde v,
  \end{align*}
  where we have used the constraint to obtain the second equality. We
  find
  \begin{align*}
    - \frac1{\eta}\, \Big\la \eta^2 \Delta v_\eta + \nabla\cdot
    (A\nabla w_\eta) , \tilde v \Big\ra = 0,\ \ \forall\,\tilde v\in
    X\ ,
  \end{align*}
  which is the weak form of \eqref {eq:orig-imag}. We use here that
  $\tilde v$ can be any element of $\HH$ with compact support, since
  an associated $\tilde w\in \HH$ can be obtained as the solution of a
  Poisson problem.  As a solution of \eqref
  {eq:orig-real}-\eqref{eq:orig-imag}, the pair $(v_\eta,w_\eta)$
  defines through $u_\eta\equiv v_\eta+i\eta^{-1}w_\eta$ a solution of
  the original problem \eqref {eq:original-copy}.

  The uniqueness is a consequence of the fact that the original
  problem \eqref {eq:original-copy} possesses a unique solution. This
  can be seen from the Lax-Milgram Lemma.  We introduce a sesquilinear
  form $b(\cdot,\cdot):X \times X \to\C$ defined by
  \begin{align*}
    b(\tilde u_1, \tilde u_2) := -i \int_{\R^2} a_\eta \nabla \tilde
    u_1 \overline{\nabla\tilde u_2}\,.
  \end{align*}
  The form $b(\cdot,\cdot)$ is coercive on $X$
  \[ \Re b(\tilde u, \tilde u) \ge\ C\|\tilde u\|_X^2\ ,\] since the
  imaginary part of $a_\eta$ is strictly positive and by the
  Poincar\'e inequality. Existence and uniqueness of a weak solution
  $u_\eta\in \HH$ solution follows from the Lax-Milgram Lemma.

  \smallskip

  \noindent\emph{Point 3.} 
  The energy equality \eqref {eq:orig-primal} was already observed in
  \eqref{EI}. 

  The proof of Lemma \ref{lem:primal} is complete.
\end{proof}

\subsection{The dual variational problem}

For fixed $f\in H^{-1}(\R^2)$, we introduce the dual energy
\begin{equation}
  J_\eta(v,\psi) := \int f \psi 
  - \frac{\eta}{2} \int |\nabla v|^2 - \frac{\eta}{2} \int |\nabla \psi|^2\ ,
  \label{Jdef}
\end{equation}
defined for $(v,\psi)\in\HH$. The dual variational problem is given by
\begin{equation}
  \label{eq:dual}
  \begin{split}
    &\textrm{maximize}\ J_\eta(\tilde v,\tilde \psi)\ 
    \textrm{over all pairs}\  (\tilde v,\tilde \psi)\\
    &\text{which satisfy the PDE constraint } \nabla\cdot (A\nabla
    \tilde\psi) + \eta\Delta \tilde v = 0.
  \end{split}
\end{equation}

The following lemma establishes that the dual variational problem is
also equivalent to the original complex equation.
 
\begin{lemma}\label{lem:dual}
  Let $f\in H^{-1}(\R^2)$ be a fixed real-valued source with compact
  support and with vanishing average. Then the dual variational
  problem \eqref {eq:dual} is equivalent to the original problem
  \eqref {eq:original-copy} with energy \eqref {eq:complex-energy} in
  the following sense.  
  \begin{enumerate}
  \item The supremum 
    \begin{equation} \label{supremum} \sup \left\{\ J_\eta(\tilde
        v,\tilde \psi)\ \middle\vert \ (\tilde v,\tilde\psi)
        \in\HH\times\HH,\ \ \nabla\cdot (A\nabla \tilde \psi) + \eta
        \Delta \tilde v = 0 \right\}
    \end{equation}
    is attained at a pair $(v_\eta,\psi_\eta)\in\HH\times\HH$.
  \item The maximizing pair, $(v_\eta,\psi_\eta)$, is unique up to an
    additive constant. The function $u_\eta := v_\eta + i \psi_\eta$
    is the unique (up to constants) solution of the original problem
    \eqref {eq:original-copy}.
  \item For the solutions, the energies coincide,
    \begin{equation} \label{eq:orig-dual}
      E_\eta(u_\eta) = J_\eta(v_\eta,\psi_\eta)\,.
    \end{equation}
  \end{enumerate}
\end{lemma}

\noindent\emph{Remark}. The lemma implies
\begin{equation}
  \label{eq:lem-dual-conseq}
  E_\eta(u_\eta) \ge J_\eta(\tilde v,\tilde \psi)
\end{equation}
for {\em every} pair $(\tilde v,\tilde \psi)$, which satisfies the PDE
constraint of \eqref{eq:dual}.  We shall use inequality
\eqref{eq:lem-dual-conseq} to establish our results on resonance.

\begin{proof}
  \noindent\emph{Point 1.}
  The existence of a maximizing pair $(v_\eta,\psi_\eta)$ for the
  variational problem \eqref{supremum} follows from the concavity of
  $J_\eta$ (convexity of $-J_\eta$), by arguments analogous to those
  given above for the primal variational problem.
  
  \smallskip
  \noindent\emph{Point 2.}
  At a maximizer,
  $(v_\eta,\psi_\eta)$, one has for all $\tilde v, \tilde \psi \in \HH
  \cap L^2(\R^2)$ satisfying the PDE constraint $\nabla\cdot (A\nabla
  \tilde\psi) + \eta\Delta \tilde v = 0$, that
  \begin{equation*}
    \partial_\tau J_\eta\Big(v_\eta+\tau \tilde v,\psi_\eta + \tau
    \tilde\psi\Big)\Big|_{\tau=0}\ =\ 0.
  \end{equation*}
  For the energy $J_\eta$, this relation provides
  \begin{align}
    \int f \tilde\psi - \eta \int \nabla v_\eta\cdot \nabla\tilde v -
    \eta \int \nabla \psi_\eta\cdot \nabla \tilde \psi\ =\
    0. \label{EL1}
  \end{align}
  Using the PDE constraint $\nabla\cdot (A\nabla \tilde \psi) + \eta
  \Delta \tilde v = 0$ to replace $\tilde v$, we find that \eqref{EL1}
  is equivalent to
  \begin{align*}
    0 &= \int f \tilde\psi + \int \nabla v_\eta\cdot A\nabla \tilde\psi
    - \eta \int \nabla \psi_\eta\cdot \nabla \tilde \psi\\
    &= \Big\la f - \nabla\cdot (A\nabla v_\eta) + \eta \Delta
    \psi_\eta, \tilde\psi \Big\ra \,.
  \end{align*}
  We conclude that the pair $(v_\eta,w_\eta) :=
  (v_\eta,\eta\psi_\eta)$ is a weak solution of \eqref
  {eq:orig-real}-\eqref{eq:orig-imag} and, thus, that $u_\eta :=
  v_\eta+ i \psi_\eta$ is a solution of $\nabla\cdot\left(a_\eta\nabla
    u_\eta\right) = f$ on $\R^2$.  Uniqueness follows again from the
  fact that $u_\eta$ is unique up to constants.

  \smallskip
  \noindent\emph{Point 3.}
  Regarding the energy equality, we calculate
  \begin{align*}
    E_\eta(u_\eta) - J_\eta(v_\eta,\psi_\eta) &= \frac{\eta}{2} \int
    |\nabla u_\eta|^2 - J_\eta(v_\eta,\psi_\eta)
    = \eta \int |\nabla v_\eta|^2 + \eta \int |\nabla \psi_\eta|^2 - \int f \psi_\eta\\
    &= - \int \nabla v_\eta\cdot A\nabla \psi_\eta - \Big\la \psi_\eta,
    -f + \nabla\cdot (A\nabla v_\eta)\Big\ra- \int f \psi_\eta = 0.
  \end{align*}
  This concludes the proof of Lemma \ref{lem:dual}.
\end{proof}

\section{Resonance results}
\label{sec.general-res-loc}

As discussed in the introduction, we consider configurations of the
following type:
\begin{enumerate}
\item The coefficients $a_\eta(x)$ and $A(x)$ are defined by
  \eqref{eq:a_eta}-\eqref{eq:A-def} with core $\Sigma\subset B_1(0)$
\item The source, $f(x)$, is concentrated at a distance $q>0$ from the
  origin and is taken of the form $f = F\, \calH^1\lfloor \del B_q(0)$
  as in \eqref {eq:structure-f-at-q}.
\end{enumerate}
We seek conditions on configurations, which ensure resonance or
non-resonance in the sense of Definition \ref{def:resonance}.

We explore the resonance properties of a configuration as follows. To
prove resonance we use the dual variational principle, exploiting
\eqref {eq:lem-dual-conseq}.  It suffices to construct, given $\eta =
\eta_j\to 0$, a sequence of comparison functions $(v_\eta, \psi_\eta)$
that satisfy the constraint of \eqref {eq:dual} and that have
unbounded energies $J_\eta(v_\eta, \psi_\eta)$.  To prove
non-resonance we use the primal variational principle, exploiting
\eqref {eq:lem-primal-conseq}.  It suffices to construct, given $\eta
= \eta_j\to 0$, a sequence of comparison functions $(v_\eta, w_\eta)$
that satisfy the constraint of \eqref {eq:primal} which have bounded
energies $I_\eta(v_\eta, w_\eta)$.

In this section, we show resonance in both radial and non-radial
settings. The non-resonance results will be presented in
Section~\ref{sec.nonres.radial} for radial cases and
Section~\ref{sec.eccentric} for a non-radial geometry.

\smallskip The basis of construction of trial functions  is the family
of {\it perfect plasmon waves}:

\begin{remark}\label{rem:perfect-plasmon}
  Consider the PDE for functions $\psi : \R^2\to \R$,
  \begin{align}
    &\nabla\cdot\left(A\nabla \psi\right)\ =\ 0,\nonumber\\
    &\nabla \psi(x)\ \to\ 0,\ \ {\rm as}\ \ |x|\to\infty
  \end{align}
  where
  \begin{equation}
    A(x) = \left\{ 
      \begin{array}{cl} 
        -1\quad & |x|\le R\\
        +1\quad & |x|> R
      \end{array} \right.   .
    \label{coreless}\end{equation}
  For any $k\ge1$, there is a non-trivial solution $\psi = \hat\psi_k$
  which achieves its maximum at a point with $|x|=R$, given by:
  \begin{equation}
    \label{eq:choice-psi-emptyset}
    \hat\psi_k(x) := 
    \begin{cases}
      r^k \cos(k\theta)\quad &\text{ for } |x|<R,\\
      r^{-k} R^{2k} \cos(k\theta)\quad &\text{ for } |x|\ge R .
    \end{cases}
  \end{equation}
  We call such functions {\it perfect plasmon waves}. Notice that
  \begin{align*}
    \int |\nabla \hat\psi_k|^2 = \int_{|x|=R} \hat \psi_k
    \left[\frac{\partial \hat\psi_k}{\partial r} \right] = 2 \pi k
    R^{2k}\, .
  \end{align*}
\end{remark}

Since our proofs rely on these perfect plasmon waves, we always use
(except for Proposition \ref {prop:holomorphic-trafo}) the circular
outer shell boundary $\del B_R(0)$. For the same reason, our technique
is restricted to the two-dimensional setting. (For an explanation why
it does not extend to 3 or more dimensions, see Appendix
\ref{perfectplasmons?}).

Since perfect plasmon waves are given in terms of Fourier harmonics, it is
natural to expand arbitrary sources $F\in L^2(\del B_q(0))$ with
vanishing average in Fourier series. We will represent an arbitrary
source by superposition of the elementary sources parametrized by
harmonic index, $k\in \N$, and source-distance, $q$:
\begin{equation}
  \label{eq:form-f-special}
  f_k^q = \cos(k\theta)\, \calH^1\lfloor \del B_q(0).
\end{equation}

\subsection{Resonance in the radial case }

\begin{proposition}[No core $\implies$ Resonance for sources at any
  distance $q$]\label{prop:nocore}
  Assume no core, $\Sigma=\emptyset$, so that $a_\eta(x) = A(x)+i\eta$
  where $A(x)$ is given by \eqref{coreless}.  Let $f=F\calH^1\lfloor
  \del B_q(0)$ with $0 \neq F:\del B_q(0)\to\mathbb{R}$ be a source at
  a distance $q > R$. Then the configuration is resonant,
  \textit{i.e.}  $E_{\eta}(u_{\eta}) \to \infty$ as $\eta \to 0$.
\end{proposition} 

\begin{proof}
  We fix the radii $R$ and $q$ and consider an arbitrary sequence
  $\eta = \eta_j \to 0$.  We write the source as $f =
  \sum_{k=1}^\infty \alpha_k f^q_k$, where $f^q_k$ is defined in
  \eqref {eq:form-f-special}.  Since $F\neq 0$, there exists some
  $k\ge1$, such that $\alpha_k \neq 0$.
  Our aim is to find a sequence $(v_\eta, \psi_\eta)$, satisfying the
  constraint $\nabla\cdot (A\nabla \psi_\eta) + \eta\Delta v_\eta = 0$
  of \eqref {eq:dual} and such that $J_\eta(v_\eta,\psi_\eta) \to
  \infty$. We choose
  \begin{align}
    v_\eta(x)\ &\equiv\ 0\\
    \psi_\eta(x) &:= \lambda_\eta \hat\psi_{k}(x)\,,
  \end{align}
  where $\hat\psi_{k}$ is the perfect plasmon wave of
  \eqref{eq:choice-psi-emptyset} and $\lambda_\eta\in \R$ is to be
  chosen below. The pair $(v_\eta,\psi_\eta)$ satisfies the constraint
  \eqref {eq:dual}. Using \eqref {eq:lem-dual-conseq}, the definition
  of $J_\eta$, the hypothesis $q>R$, and the orthogonality of Fourier
  harmonics, we obtain
  \begin{align*}
    E_\eta(u_\eta) &\ge J_\eta(v_\eta,\psi_\eta) =
    J_\eta(0,\psi_\eta)
    = \int f\cdot \psi_\eta - \frac{\eta}{2} \int |\nabla \psi_\eta|^2\\
    &= \int_{\del B_q(0)} \alpha_{k} \cos(k\theta)\cdot \lambda_\eta
    q^{-k} R^{2k} \cos(k\theta)
    -  \frac{\eta}{2} |\lambda_\eta|^2 \int |\nabla \hat\psi_k|^2\\
    &\ge 
    \pi\, q \alpha_k \lambda_\eta\, q^{-k} R^{2k}
    - C_0\,  \left(\eta  |\lambda_\eta|^2\right)\, k R^{2k}\,.
  \end{align*} 
  Choosing $\lambda_\eta \to\infty$ with $\eta |\lambda_\eta|^2\to 0$
  we obtain $E_\eta(u_\eta) \to \infty$ for $\eta\to 0$.
\end{proof}

\subsection{Resonance in the non-radial case}

Our first observation for non-radial geometry regards a variant of
Proposition \ref {prop:nocore}. We consider the index $A = -1$ in a
domain $D\subset \R^2$, which is similar, but not identical to the
ball $B_R(0)$.  

Let $\Phi : \C\to\C$ be a holomorphic function. For three radii $R < q
< s$ we introduce the three domains $D_R := \Phi(B_R(0))$, $D_q :=
\Phi(B_q(0))$ and $D_s := \Phi(B_s(0))$. We assume that $\Phi$ is
bijective on the largest ball, $\Phi|_{B_s(0)} : B_s(0) \to D_s$ has a
inverse $\Psi:D_s\to B_s(0)$. 

\begin{proposition}[Resonance for non-radial structures without core]
  \label{prop:holomorphic-trafo}
  Fix radii $1 < R < q < s$, holomorphic maps $\Phi$ and
  $\Psi=\Phi^{-1}$ as above. Assume $s > q^2/R$. Consider the equation
  \eqref {eq:original}:
  \[ \nabla\cdot\left(a_\eta \nabla u_\eta\right)\ =\ f\]
  where $a_\eta(x)=A(x)+i\eta$ and 
  \begin{equation}
    A(x) = \left\{ 
      \begin{array}{cl} 
        -1\quad &  x\in D_R\\
        +1\quad &  x\notin D_R
      \end{array} \right.  
    \label{A-def-general}
  \end{equation}   
  Then there exists a source $f = F\, \calH^1\lfloor \del D_q$, where
  $F\in L^2(\del D_q)$, such that the configuration is resonant, {\it
    i.e.}\ $E_\eta(u_\eta)\to \infty$ for $\eta\to 0$.
\end{proposition}

\begin{proof}
  We proceed as in Proposition \ref {prop:nocore} and exploit the dual
  variational principle.  Our aim is to construct a sequence $(v_\eta,
  \psi_\eta)$, satisfying the constraint $\nabla\cdot (A\nabla
  \psi_\eta) + \eta\Delta v_\eta = 0$ of \eqref {eq:dual} and such
  that $J_\eta(v_\eta,\psi_\eta) \to \infty$.  However in this case,
  due to the coupling of Fourier harmonics in a non-radial geometry,
  we cannot restrict ourselves to a harmonic of fixed index, $k$.

  We start the construction from the perfect plasmon waves
  $\hat\psi_k$ of \eqref {eq:choice-psi-emptyset}. They are mapped
  with $\Phi$ to functions $\tilde\psi_k : D_s \to \R$, $\tilde\psi_k
  := \hat\psi_k\circ \Psi$. We note that these functions are harmonic
  in $D_s\setminus \del D_R$. Regarding the jump of the normal
  derivative along $\del D_R$, we can calculate as follows. The normal
  vector $e_r$ to $\del B_R(0)$ is mapped to the (not normalized)
  normal vector $\nu = D\Phi\cdot e_r$ of $\del D_R$.  In a point
  $x\in \del D_R$ the matrix $D\Phi$ can be represented by a complex
  number $M\in \C$, and we can calculate $\la \nu, \nabla \tilde\psi_k
  \ra = \la M e_r , ((\nabla \hat\psi_k)^T \cdot M^{-1})^T \ra = \la M
  e_r, (M^{-1})^T \nabla \hat\psi_k \ra = \la e_r, \nabla \hat\psi_k
  \ra$. Thus, $\tilde\psi_k$ solves
 \begin{equation}
 \nabla\cdot \left(A(x)\nabla \tilde\psi_k(x) \right) = 0,\text{ for }  \ x\in D_s\ .
 \label{tpsi-is-A-harmonic}
 \end{equation}

  In order to define functions on all of $\R^2$, we introduce a smooth
  cut-off function $h: \R^2\to [0,1]$. By our assumption on the radii,
  we can choose a number $Q\in (q^2/R, s)$ and $D_Q =
  \Phi(B_Q(0))$.  Let $h$ be such that  $h \equiv 1$
  on $D_Q$ and $h\equiv 0$ on $\R^2\setminus D_s$. 
  
  We now construct a comparison function of the form
  \begin{equation}
  \psi_\eta(x) := \lambda_\eta\, h(x)\, \tilde\psi_k(x)\,,\label{psi-eta}
  \end{equation}
  where $\lambda_\eta$ and $k=k(\eta)$ are to be chosen below.
   Once $\psi_\eta$ is determined, the function $v_\eta$ is
  chosen as the bounded solution of 
  \begin{equation}
    \eta\Delta v_\eta(x) = - \nabla\cdot (A(x)\nabla \psi_\eta(x))
    = - \nabla\cdot (A(x)\nabla \psi_\eta(x))\, {\bf 1}_{D_s\setminus D_Q}(x)\,.
    \label{v-eta}
  \end{equation}
  The latter equality holds by \eqref{tpsi-is-A-harmonic} since
  $\nabla h(x)\ne0$ precisely on $D_s\setminus D_Q$.  The pair
  $(v_\eta,\psi_\eta)$ satisfies, by construction, the constraint
  \eqref {eq:dual}. It only remains to verify
  $J_\eta(v_\eta,\psi_\eta) \to \infty$.

  To motivate our choice of $\lambda_\eta$ and $k(\eta)$, we first
  compute the contributions to the energy, $J_\eta$, of the functions
  $v_\eta$ and $\psi_\eta$.  In the following calculations, $C$
  denotes a constant that is independent of $k$ and $\eta$; $C$
  depends on the geometry of the structure and its value may change
  from one line to the next. For the contribution from $v_\eta$ we
  find
  \begin{align}
    \eta\int |\nabla v_\eta|^2 
    &\le C\, \frac{1}{\eta}\int_{D_s\setminus D_Q}|\nabla\psi_\eta|^2\nonumber\\
    &\le C\, \frac{\lambda_\eta^2}{\eta}\int_{D_s\setminus D_Q}\
    |\tilde\psi_k|^2+|\nabla\tilde\psi_k|^2 \le C\, 
    \frac{\lambda_\eta^2}{\eta} k \left(\frac{R^2}{Q}\right)^{2k}\ .
    \label{energy-veta}
  \end{align}
  The contribution to the energy, $J_\eta$, from $\psi_\eta$ is
  \begin{align}
    \eta \int |\nabla \psi_\eta|^2 \le C\, \eta\lambda_\eta^2\left(
      kR^{2k} +\frac{1}{k}\left(\frac{R^2}{Q}\right)^{2k} \right)\
    \le C\, \eta k \lambda_\eta^2R^{2k}\ .
    \label{energy-psieta}
  \end{align}
  It remains to evaluate the first term of the energy, $J_\eta$, for a
  source $f=F\calH^1\lfloor \del D_q$. For some positive constant
  $c>0$ we obtain
  \begin{align}
    \int f\,\psi_\eta &= \lambda_\eta \int_{\del D_q} F \tilde\psi_k
    \ge c\,  \lambda_\eta \int_{\del B_q(0)} (F\circ\Phi) \hat\psi_k\nonumber\\
    &= c\, \lambda_\eta\ \left(\frac{R^2}{q}\right)^k\ \int_{\del B_q(0)}
    (F\circ\Phi)\cos (k\theta) \ \equiv\  c \lambda_\eta\
    \left(\frac{R^2}{q}\right)^k\ \alpha_k\,,
    \label{f-induced}
  \end{align}
  where $\alpha_k$ is a Fourier coefficient for $F\circ\Phi$.  We
  shall drive $J_\eta(v_\eta,\psi_\eta)$ to infinity by driving the
  contribution \eqref{f-induced} to infinity while keeping the
  contributions \eqref{energy-veta} and \eqref{energy-psieta} bounded.

  Balancing the upper bounds \eqref{energy-veta} and
  \eqref{energy-psieta} requires that we choose $k=k(\eta)$ such that:
  \begin{align*}
    \left(R/Q \right)^{k(\eta)} \sim \eta.
  \end{align*}
  Since $R/Q < R^2/q^2 < 1$, there is such $k(\eta)$, and $k(\eta)\to
  \infty$ as $\eta\downarrow 0$.  With this choice of $k$,
  \begin{equation}
    \eta\int |\nabla v_\eta|^2 + \eta\int |\nabla\psi_\eta|^2\ 
    \le C\, \lambda_\eta^2\, k(\eta) \left(\frac{R^3}{Q}\right)^{k(\eta)}\,.
    \label{bounded-part}
  \end{equation}
  To keep this contribution to the energy bounded we choose
  $\lambda_\eta$ so that
  \begin{equation}
    \lambda_\eta^2 = \frac{1}{k(\eta)} \left(\frac{Q}{R^3}\right)^{k(\eta)}\,.
    \label{lambda-eta}
  \end{equation}

  Finally, substitution of \eqref{lambda-eta} into \eqref{f-induced}
  we obtain, for some $c, C >0$:
  \begin{equation}
    \Big|\ J_\eta(v_\eta,\psi_\eta)\ \Big|\ 
    \ge\ \left|\ \int f \psi_\eta\ \right|\ -\ C\ 
    \ge c\, \frac{1}{\sqrt{k(\eta)}}\ \left(\frac{Q R}{q^2}\right)^{k/2}\ |\alpha_k|\ -\ C .
  \end{equation}

  We have chosen $Q$ with $Q>q^2/R$. Therefore, if we choose $F$ such
  that its Fourier coefficients, $\alpha_k$, are not decaying too
  rapidly, we have $|J_\eta(v_\eta,\psi_\eta)|\to\infty$ as
  $\eta\to0$. This completes the proof of Proposition
  \ref{prop:holomorphic-trafo}.
\end{proof}

Next, we consider a non-radial geometry with a core, $\Sigma\subset
B_1(0)$, of arbitrary shape. The following resonance result has a
proof very similar to that of Proposition~\ref{prop:nocore}.

\begin{theorem}[Any shape core, source at $q<R^*\ \implies$\ Resonance]
  \label{thm:inner-resonance}
  Let $\Sigma\subset B_1(0)$ be an arbitrary core with $\Gamma =
  \del\Sigma$ a curve of class $C^2$. Then, for every radius $R < q <
  R^* := R^{3/2}$, there exists $f$ supported at a distance $q$, such
  that the configuration is resonant.
\end{theorem}

\noindent \textit{Remark.} Our proof actually gives a slightly
stronger result. Not only does there exist $f$, supported at
$q\in(R,R^*)$, $E_\eta\to\infty$ as $\eta\to0$, but furthermore the
divergence of the energy occurs for every $f$ having high frequency
components of sufficiently large amplitude.

\begin{proof}
  We fix $R < q < R^*$ and a sequence $\eta = \eta_j \to 0$. We
  consider the source function $f = \sum_{k=1}^\infty \alpha_k f_k^q$
  with $f_k^q$ as in \eqref {eq:form-f-special}.  Our aim is to
  construct a sequence $(v_\eta, \psi_\eta)$, satisfying $\nabla\cdot
  (A\nabla \psi_\eta) + \eta\Delta v_\eta = 0$ of \eqref {eq:dual}
  with $J_\eta(v_\eta,\psi_\eta) \to \infty$.

  As in the proof of Proposition \ref {prop:nocore} our sequence of
  trial functions is built using perfect plasmon waves; as before we
  choose the harmonic index $k=k(\eta)$ to depend on $\eta$ and set
  \begin{align*}
    \psi_\eta(x) = \lambda_\eta \hat\psi_{k(\eta)}(x)\, . 
  \end{align*}
  The numbers $k = k(\eta)\in \N$ and $\lambda_\eta\in \R$ will be
  chosen below.

  The function $\psi_\eta$ is not $A$-harmonic along the core
  interface $\del\Sigma\subset B_1(0)$. In order to satisfy the constraint we
  therefore define $v_\eta$ as the solution of $\eta\Delta v_\eta =
  -\nabla\cdot (A\nabla \psi_\eta)$. By elliptic estimates
  \[ \eta\|\nabla v_\eta\|_{L^2(\mathbb{R}^2)}^2\le C\, \eta^{-1}\
  \|\nabla\cdot (A\nabla \psi_\eta)\|_{H^{-1}(\mathbb{R}^2)}^2\le C\,
  \eta^{-1}\ \lambda_\eta^2\, k(\eta)\,.\]

  It remains to calculate the energy $J_\eta(v_\eta,\psi_\eta)$. We
  choose $k = k(\eta)$ to be the smallest integer with $R^{-k} < \eta$
  and note that $R^{-k+1} \ge \eta$ holds. Exploiting
  \eqref{eq:lem-dual-conseq}, we obtain, for some $c_0 >0$
  \begin{align*}
    E_\eta(u_\eta) &\ge J_\eta(v_\eta,\psi_\eta) = \int f \psi_\eta -
    \frac{\eta}{2} \int |\nabla \psi_\eta|^2
    -  \frac{\eta}{2} \int |\nabla v_\eta|^2\\
    &\ge c_0 \alpha_k \lambda_\eta q^{-k} R^{2k}
    - C\eta k \lambda_\eta^2 R^{2k} - C \eta^{-1} k\, \lambda_\eta^2\\
    &= \lambda_\eta R^k \left(\ c_0\alpha_k\left(\frac{R}{q}\right)^k
      - C \lambda_\eta k (\eta R^k) - C\frac{1}{(\eta R^k)}
      \lambda_\eta k \right)
    \end{align*}
    The choice of $k$ with $1 < \eta R^k \leq R$ ensures that the last
    two contributions are of comparable order.
    We find, for some $C_0 > 0$,
    \begin{align*}
      E_\eta(u_\eta)\ &\ge \ \lambda_\eta R^{k(\eta)}
      \left( c_0\alpha_{k(\eta)}\left(\frac{R}{q}\right)^{k(\eta)} - C_0
        \lambda_\eta k(\eta) \right).
    \end{align*}
    We choose $\lambda_\eta$ such that the right hand side is
    positive, specifically
    \[\ \lambda_\eta\ = \frac{1}{2C_0 k(\eta)}\ c_0\alpha_{k(\eta)}
    \left(\frac{R}{q}\right)^{k(\eta)}\,.\] Thus,
    \begin{align*}
      E_\eta(u_\eta)\ &\ge \ \lambda_\eta R^{k(\eta)}\left(\ \frac{1}{2}\
        c_0\alpha_{k(\eta)}\left(\frac{R}{q}\right)^{k(\eta)}\
      \right)\ =\ \frac{1}{4C_0 k(\eta) }\ (c_0\alpha_{k(\eta)})^2\
      \left(\frac{R^3}{q^2}\right)^{k(\eta)}.
    \end{align*}
    By assumption, $q$, the location of the source satisfies $q<R^*$,
    or equivalently $q^2 < R^3$.  To ensure that $E_\eta(u_\eta)\to
    \infty$ it suffices to assume that the sequence of Fourier
    coefficients $(\alpha_k)_k\in l^2(\N)$ decays sufficiently
    slowly. In particular, if $(\alpha_k)_k$ decays algebraically we
    have $E_\eta(u_\eta)\to \infty$. This completes the proof of
    Theorem \ref{thm:inner-resonance}.
\end{proof}

\section{Non-resonance in the radial case}
\label{sec.nonres.radial}

In this section, we use the primal variational principle
\eqref{eq:primal} to show non-resonance in the radial case for sources
located at distance larger than the critical radius $R^{\ast}$.

\begin{proposition}[Non-resonance beyond $R^*$ in the radial case]
  \label{prop:spher-incl}
  Consider $a_\eta(x)$ of \eqref {eq:a_eta}--\eqref {eq:A-def}, with
  the radial concentric arrangement $\Sigma = B_1(0)\subset
  B_R(0)$. Assume $f= F \calH^1\lfloor \del B_q(0)$, $F \in L^2(\del
  B_q(0))$.  Then, for any $q > R^* := R^{3/2}$, the configuration is
  non-resonant.
\end{proposition}

\begin{proof} 
  Expanding in Fourier series we have $F = \sum_{k\ge 1} \alpha_k
  \cos(k\theta) + \sum_{k\ge 1} \beta_k \sin(k\theta) \equiv F_\even +
  F_\odd$. It suffices to prove that $f_\even = F_\even \calH^1\lfloor
  \del B_q(0)$ and $f_\odd = F_\odd \calH^1\lfloor \del B_q(0)$ are
  non-resonant.  We give the argument for $f_\even$; the argument for
  $f_\odd$ is the same. Accordingly, we consider from now on $f =
  \sum_k \alpha_k f_k$, where $f_k$ is given by \eqref
  {eq:form-f-special} and $(\alpha_k)_{k\ge1} \in l^2(\N;\R)$; we
  suppress here the superscript $q$ of $f_k^q$.  We will construct
  test functions in Step 1 and compute their energy in Step 2 to prove
  the Proposition.

  \medskip 
  \noindent \emph{Step 1. Construction of comparison functions.}  To
  prove non-resonance, we use the primal variational problem \eqref
  {eq:primal}. Consider a fixed sequence $\eta = \eta_j$ tending to
  zero. We shall construct $(v_\eta, w_\eta)$, satisfying the
  constraint
  \begin{equation}
    \nabla\cdot \left( A\nabla v_\eta\right)\ -\ \Delta w_\eta\ =\ f
    \label{constraint}
  \end{equation}
  such that the energy along this sequence, $I_\eta(v_\eta, w_\eta)$
  remains bounded. Our strategy is to decompose the source $f$ into a
  low frequency part and a high frequency part as
  \begin{equation}
    \label{eq:low-high-f}
    f = f^\low +f^\high,\qquad 
    f^\low := \sum_{k=1}^{k^*} \alpha_k f_k,\qquad
    f^\high := \sum_{k=k^* +1}^{\infty} \alpha_k f_k\,,
  \end{equation}
  where $k^*$ is chosen to depend on $\eta$. Later we will choose $k^*
  = k^*(\eta)$ to be the smallest integer for which $R^{-k^*} > \eta$.

  Our approach, to be discussed in detail below, is to solve the
  constraint equation \eqref{constraint} in the form: $v_\eta =
  v_\eta^\low + v_\eta^\high$ where
  \begin{align}
    & \textrm{$v_\eta^\low$ satisfies}\ \  
    \nabla\cdot (A\nabla v_\eta^\low) = f^\low\label{vlow}\\
    & \textrm{$v_\eta^\high$ satisfies}\ \ \nabla\cdot (A\nabla
    v_\eta^\high) \Big|_{\del B_q(0)}
    = f^\high\label{vhigh}\\
    & \textrm{$w_\eta$ satisfies}\ \ -\Delta w_\eta = - \nabla\cdot
    (A\nabla v_\eta^\high) + f^\high\label{w_eta}
  \end{align}
  This construction yields $(v_\eta, w_\eta)$, which satisfies the
  constraint \eqref{constraint} of the primal problem
  \eqref{eq:primal}.  Furthermore, we shall see that with an
  appropriate choice of cutoff $k^{\ast}=k^{\ast}(\eta)$ in
  \eqref{eq:low-high-f}, $I_\eta(v_\eta, w_\eta)$ remains bounded as
  $\eta\to0$.  As in our analysis of resonance, we shall make strong
  use of the perfect plasmon waves.
  
  \medskip

  \noindent\emph{Step 1a. Construction of $v_\eta^\low$.} The
  function $v_\eta^\low$ is pieced together using variants of the
  perfect plasmon waves.
    \begin{equation}
    \label{eq:hat-v-k-radial}
    \hat v_k(x) :=
    \begin{cases}
      r^k \cos(k\theta)\quad &\text{ for } |x| \le 1,\\
      r^{-k} \cos(k\theta)\quad &\text{ for } 1 < |x| \le R,\\
      r^k R^{-2k} \cos(k\theta)\quad &\text{ for } R < |x|\le q,\\
      r^{-k} (q/R)^{2k} \cos(k\theta)\quad &\text{ for } q <|x|.
    \end{cases}
  \end{equation}
  We note that $\hat v_k$ has the following properties.
  \begin{enumerate}
  \item $\hat v_k$ is continuous on all $\mathbb{R}^2$
  \item $\hat v_k$ satisfies $\nabla\cdot (A\nabla \hat v_k) = 0$ for
    $x\in\R^2\setminus \del B_q(0)$.
  \item Along $\del B_q(0)$, $\hat v_k$ has a jump in its normal flux:
  \begin{align*}
    [ \nu\cdot \nabla \hat v_k ]_{\del B_q(0)}
    = \left\{ \frac{-k}{q} q^k R^{-2k} - \frac{k}{q} q^k R^{-2k}\right\} 
    \cos(k\theta)
    = -\frac{2k}{q} q^k R^{-2k} \cos(k\theta)\,.
  \end{align*}
  \end{enumerate}
  Therefore, an appropriate constant multiple $\lambda_k\hat v_k$ will
  satisfy
  \begin{equation*}
    \nabla\cdot (A\nabla \lambda_k\hat v_k) = \alpha_k f_k\ \ {\rm
      on}\ \mathbb{R}^2.
  \end{equation*}
  In order to satisfy this relation, we must choose $\lambda_k$ with
  \begin{equation*}
    \lambda_k\cdot \left(-\frac{2k}{q} q^k R^{-2k}\right)\ =\
    \alpha_k\,.
  \end{equation*} 
  We therefore set
  \begin{equation}
    \label{eq:v-eta-low}
    v_\eta^\low := \sum_{k = 1}^{k^*} \lambda_k\, \hat v_k\,,\quad
    \text{with }\
    \lambda_k := -\alpha_k \frac{q}{2k} q^{-k} R^{2k}\,.
  \end{equation}
  This function satisfies \eqref{vlow},
  \begin{align*}
    \nabla\cdot (A\nabla v_\eta^\low) 
    &= [ \nu\cdot \nabla v_\eta^\low]_{\del B_q(0)}\ \calH^1\lfloor \del B_q(0)
    = \sum_{k\le k^*} \lambda_k\, 
    \left(-\frac{2k}{q} q^k R^{-2k} \cos(k\theta) \right)
    \calH^1\lfloor \del B_q(0)\\
    &= \sum_{k\le k^*} \alpha_k \cos(k\theta) \calH^1\lfloor \del B_q(0)
    = f^\low\,.
  \end{align*}

  \smallskip 

  \noindent\emph{Step 1b. Construction of $v_\eta^\high$ and
    $w_\eta$.} The function $v_\eta^\high$ is constructed from the
  elementary plasmon waves $\hat V_k$ for the radius $q$. The
  functions are not tuned to solve $\nabla\cdot(A\nabla v)=0$ on $\del
  B_1(0)$ or $\del B_R(0)$, but they are small along these curves
  (compared to their maximal values). We set
  \begin{equation}
    \label{eq:hat-V-k-radial}
    \hat V_k(x) :=
    \begin{cases}
      r^k \cos(k\theta)\quad &\text{ for } |x| \le q,\\
      r^{-k} q^{2k} \cos(k\theta)\quad &\text{ for } q <|x|.
    \end{cases}
  \end{equation}
  Recall $A(x)=1$ in a neighborhood of $|x|=q$, so the jump in the
  normal flux on $\del B_q(0)$ is
 \[ \left[ \nu\cdot A\nabla \hat V_k\right]\Big|_{\del B_q(0)}\ =\ (-2k/q) q^k
  \cos(k\theta)\ .\]
  Therefore if we set
  \begin{equation}
    \label{eq:v-eta-high}
    v_\eta^\high := \sum_{k > k^*} \lambda_k\, \hat V_k\,,\qquad
    \lambda_k := -\alpha_k \frac{q}{2k} q^{-k}\,,
  \end{equation}
  it follows that \eqref {vhigh} is satisfied:
  \begin{align*}
    \nabla\cdot (A\nabla v_\eta^\high) \Big|_{\del B_q(0)} =
    f^\high\,.
  \end{align*}
  We emphasize that $v_\eta^\high$ is not a solution on all of $\R^2$
  due to normal flux jumps at $|x|=1$ and $|x|=R$.  Since \eqref
  {vlow} is satisfied, the constraint \eqref{constraint} is equivalent
  to \eqref {w_eta},
   \begin{equation}\label{Laplacew}
     \begin{aligned}
       &-\Delta w_{\eta}\ =\ - \nabla\cdot (A\nabla v_\eta^\high) + f^\high\\
       &\quad = -\sum_{k > k^*} \lambda_k\, \left[ \nu\cdot A\nabla \hat
         V_k\right]\Big|_{\del B_1(0)}\calH^1\lfloor \del B_1(0) 
       - \sum_{k > k^*} \lambda_k\,   
       \left[ \nu\cdot A\nabla \hat V_k\right]\Big|_{\del B_R(0)}
       \calH^1\lfloor \del B_R(0)\,.
     \end{aligned}
   \end{equation}
   We use this equation to define $w_\eta$.  

  \medskip
  \noindent\emph{Step 2. Calculation of energies.}  It remains to
  calculate the energy $I_\eta(v_\eta, w_\eta)$, for the above choice
  of $v_\eta = v_\eta^\low + v_\eta^\high$ and $w_\eta$. It is in this
  step that we choose the low-high frequency cutoff,
  $k^{\ast}=k^{\ast}(\eta)$ to ensure that $I_\eta(v_\eta, w_\eta)$
  remains uniformly bounded as $\eta\to0$. Once we verify the
  boundedness of this sequence of energies, the non-resonance property
  of Definition \ref {def:resonance} follows from \eqref
  {eq:lem-primal-conseq}.
 
  \smallskip 
  \noindent\emph{Step 2a. Energy of $v_\eta$.} By the triangle
  inequality we can bound the energies of $v_\eta^\low$ and
  $v_\eta^\high$ separately.  Furthermore,  orthogonality of
  Fourier modes  implies  for $v_\eta^\low$:
  \begin{align}\label{vlow-est} 
    \eta \int |\nabla v_\eta^\low|^2 = \eta \sum_{k\le k^*}
    |\lambda_k|^2\, \int |\nabla \hat v_k|^2 \le C \eta \sum_{k\le
      k^*} |\alpha_k|^2\, \left(\frac{R^2}{q}\right)^{2k}\, \max
    \Big\{ 1, \left(\frac{q}{R^2}\right)^k\Big\}^2\,.
  \end{align} 
  For the case where $q \ge R^2$, we obtain
  \begin{align*}
    \eta \int |\nabla v_\eta^\low|^2
    \le C \eta \sum_{k\leq k^{\ast}} |\alpha_k|^2 
    \le C\eta\,,
  \end{align*}
  which is obviously bounded. The case where $R^*<q<R^2$ is more
  subtle. We note here that estimate \eqref{vlow-est} simplifies in
  this case to
  \begin{align}\label{v-low-1}
    \eta \int |\nabla v_\eta^\low|^2 \le C \eta \sum_{k\leq k^{\ast}}
      |\alpha_k|^2\,   \left(\frac{R^2}{q}\right)^{2 k^{\ast}}.
  \end{align}
  We will come back to this bound soon with a specific choice of
  $k^{\ast}$.

  The energy of $v_\eta^\high$ is easier to control:
  \begin{align}
    \eta \int |\nabla v_\eta^\high|^2 
    \le C \eta \sum_{k> k^*} |\lambda_k|^2\, \int |\nabla \hat V_k|^2
    \le C \eta \sum_{k} |\alpha_k|^2\, 
    \le C.\label{v-high-est}
  \end{align}

  \smallskip 
  \noindent\emph{Step 2b. Energy of $w_\eta$.}  Next we study
  the energy of $w_\eta$. By the properties of the solution operator
  $(-\Delta)^{-1}$ acting on functions in $H^{-1}(\R^2)$ with
  vanishing average, we have
  \begin{align*}
    \frac1{\eta} \int |\nabla w_\eta|^2 \le C \frac1{\eta} \|
    \nabla\cdot (A\nabla v_\eta^\high) - f^\high \|_{H^{-1}}^2 \ \le C
    \frac1{\eta} \sum_{k> k^*} |\lambda_k|^2 R^{2k}\, k.
  \end{align*}
  The last inequality follows from \eqref{Laplacew}, which states that
  $\nabla\cdot (A\nabla v_\eta^\high) - f^\high$ is supported on
  $|x|=1$ and $|x|=R$.
    
  Now by the choice of $\lambda_k$ in \eqref{eq:v-eta-high}, we have
  $|\lambda_k| \le C |\alpha_k| k^{-1} q^{-k}$, and hence
  \begin{align}
    \frac1{\eta} \int |\nabla w_\eta|^2\ \le\ C\ 
    \sum_{k>k^*} |\alpha_k|^2\ \frac{1}{\eta}
    \left(\frac{R}{q}\right)^{2k^*} \label{w-en-est}
  \end{align}
  Balancing the right hand sides of the bounds \eqref{v-low-1} and
  \eqref{w-en-est} we choose $k^{\ast}$ so that
  \begin{equation*}
    \eta\ \left(\frac{R^2}{q}\right)^{2 k^{\ast}}
    \ \sim\ \frac{1}{\eta}\ \left(\frac{R}{q}\right)^{2k^*}\,,
  \end{equation*}
  i.e. we choose $k^*=k^*(\eta)$ to be the smallest integer with
  $R^{-k^{\ast}} < \eta$ such that
    \begin{equation}\label{eq:def-k-ast}
      \eta \leq R^{-k^{\ast}+1}, \quad\text{ and }\quad
      \frac{1}{\eta} < R^{\,k^{\ast}}.
    \end{equation}
    Combining \eqref{eq:def-k-ast} with \eqref {w-en-est} and \eqref
    {v-low-1}, we obtain
  \begin{align}
    \frac1{\eta} \int |\nabla w_\eta|^2\ \le\ C\ \sum_{k}
    |\alpha_k|^2\,
    \left(\frac{R^3}{q^2}\right)^{k^*(\eta)} \label{w-en-est-2}
  \end{align}
  and
  \begin{align}
    \eta \int |\nabla v_\eta^\low|^2 \le\ C \ \sum_{k} |\alpha_k|^2\,
    \left(\frac{R^3}{q^2}\right)^{k^*(\eta)} \label{v-low-2}
  \end{align}
  Thus, if $q > R^*=R^{3/2}$, $I(v_\eta,w_\eta)$ is bounded as
  $\eta\to0$. The proof of non-resonance is complete.
\end{proof}

\section{A non-resonance result in a non-circular geometry}
\label{sec.eccentric}

\subsection{Interaction coefficients}

In this section we use complex notation. We will use complex analysis
in order to calculate certain interaction integrals that will be of
interest for non-resonance results in non-radial geometries.

We identify $\R^2 \equiv \C$ via $(x_1,x_2) \equiv x_1 + i x_2 = z$.
The complex functions $z \mapsto z^k$ and $z \mapsto z^{-k}$ are
holomorphic on $\C$ and $\C\setminus \{0\}$; we have used the real
part of these functions before,
\begin{align*}
  \Re (z^k) = r^k\, \cos(k\theta),\quad
  \Re (z^{-k}) = r^{-k}\, \cos(k\theta),\text{ for } z = r e^{i\theta}\,.
\end{align*}
By the Cauchy-Riemann differential equations, the gradient of these
real functions can easily be calculated,
\begin{align*}
  \nabla \Re (z^k) = \vectdue{\del_{x_1}}{\del_{x_2}} \Re (z^k)
  = \vectdue{\Re \del_{x_1}}{-\Im \del_{x_1}} (z^k)
  \equiv \overline{\del_{x_1} (z^k)} = \overline{(z^k)'} = k \bar z^{k-1}\,.
\end{align*}
Accordingly, we can evaluate for $\nu \equiv z/|z|$ the normal derivative
on a sphere $\del B_r(0)$,
\begin{equation}\label{eq:complex-normal-der}
  \la  \nu , \nabla \Re (z^k) \ra = \Re \left( \frac{\bar z}{|z|} \cdot
    k \bar z^{k-1} \right) = \frac{k}{r} \Re (z^k)
  = \frac{k}{r} r^k\, \cos(k\theta)\,.
\end{equation}
A similar calculation provides for the imaginary part $\la \nu ,
\nabla \Im (z^k) \ra = \frac{k}{r} \Im (z^k)$.

Using the complex notation we define, for arbitrary radius $\rho>0$
and arbitrary center $z_0\in \C$, the function
\begin{equation}
  \label{eq:perfect-z0-plasmon}
  \Psi_m(z) := 
  \begin{cases}
    \Re ((z-z_0)^m) \quad &\text{ for } |z-z_0|< \rho,\\
    \rho^{2m} \Re ((z-z_0)^{-m}) \quad &\text{ for } |z-z_0|\ge \rho\,.
  \end{cases}
\end{equation}
This function is a perfect plasmon wave for $\del B_\rho(z_0)$ in the
sense that it is continuous and it satisfies the equation $\nabla\cdot
(A_{z_0,\rho} \nabla \Psi_m) = 0$ for the coefficient $A_{z_0,\rho} =
1$ outside $B_\rho(z_0)$ and $A_{z_0,\rho} = -1$ inside
$B_\rho(z_0)$. This is easily checked with either the above complex
calculations or the previously imployed real calculations.

Of importance will be the interaction of two perfect plasmon waves
with different centers. In particular, we need to know the following
interaction coefficients.

\begin{definition}\label{def:interaction}
  Let a radius $\rho>0$ and a center $z_0\in \C$ be such that the
  corresponding disk contains the origin, $0 \in \Sigma :=
  B_\rho(z_0)$. For wave numbers $k,m\in \N$, the {\em interaction
    coefficient} is defined as
  \begin{equation}
    \label{eq:interaction-number}
    I_{m,k} := \int_{\del\Sigma} \Re(z^{-k}) (z-z_0)^m\, d\calH^1\ \in \C\,.
  \end{equation}
\end{definition}

The integral in \eqref {eq:interaction-number} is with respect to the
Hausdorff measure; the complex number $I_{m,k}$ can therefore be
identified with the real vector whose two components are the real
integrals $\int \Re(z^{-k}) \Re (z-z_0)^m\, d\calH^1$ and $\int
\Re(z^{-k}) \Im (z-z_0)^m\, d\calH^1$. In this sense, the coefficients
$I_{m,k}$ are, up to normalizing factors, the coefficients for an
expansion of the function $\Re(z^{-k})$ in spherical harmonics of the
sphere $\del B_\rho(z_0)$.

With the help of complex analysis, the interaction coefficients can be
computed explicitly.
\begin{lemma}[Properties of $I_{m,k}$]\label{lem:interactionImk}
  The interaction coefficients for $m\ge 1$ and $k\ge 0$ are given by
  \begin{equation}
    \label{eq:Imk-explicit}
    I_{m,k} =
    \begin{cases}
      0 \quad &\text{ for } m < k,\\
      (-1)^{m-k} \pi \rho \vectdue{m-1}{m-k} z_0^{m-k} 
      \quad &\text{ for } m \ge k
    \end{cases}
  \end{equation}
  For $m=0$ we have $I_{0,k} = 2\pi\rho\, \delta_{0,k}$.

  On the circle $\del B_{\rho}(z_0)$, the function $\Re(z^{-k})$ with
  $k\ge 1$ can be expanded in spherical harmonics with center $z_0$ as
  \begin{equation}
    \label {eq:expansion-newcenter}
    \Re(z^{-k}) 
    = \Re \sum_{m\ge k} (\pi \rho^{2m+1})^{-1} \bar I_{m,k}\cdot
    (z- z_0)^m\,.
  \end{equation}

  For any number $Q>0$ we have the estimate
  \begin{equation}
    \label{eq:Q-Imk-est}
    \sum_{k\in \N} Q^k |I_{m,k}| 
    \le \pi \rho Q\, (|z_0| + Q)^{m-1}
  \end{equation}
  for every $m\in \N$.
\end{lemma}

\begin{proof}
  We note that, for $m = k = 0$, the value $I_{0,0} = 2\pi\rho$
  follows immediately from the definition of $I_{m,k}$. From now on,
  we can therefore assume $k+m \ge 1$.

  We can calculate the number $I_{m,k}$ with the help of the residue
  theorem. We decompose the integral according to $\Re(z^{-k}) =
  \frac12 z^{-k} + \frac12 \bar z^{-k}$. We calculate first the
  contribution of $\bar z^{-k}$,
  \begin{align*}
    \left( \int_{\del\Sigma} \bar z^{-k} (z-z_0)^m\, d\calH^1 \right)^{c.c.}
    &= \int_{\del\Sigma} z^{-k} (\overline{z-z_0})^m\, d\calH^1
    \stackrel{(a)}{=} 
    \rho^{2m} \int_{\del\Sigma} z^{-k}(z-z_0)^{-m}\, d\calH^1\\
    &\stackrel{(b)}{=} -i\rho\, 
    \rho^{2m} \int_{\del\Sigma} z^{-k}(z-z_0)^{-m-1}\, dz
    \stackrel{(c)}{=} 0\,,
  \end{align*}
  where the symbol $c.c.$ denotes complex conjugation. In the
  calculation above, we have used in equality (a) the fact that, for
  the argument $\vartheta\in \R$ of $(z-z_0)$, we have
  $(\overline{z-z_0})^m = \rho^m \exp(- im\vartheta) = \rho^{2m}
  (z-z_0)^{-m}$. In equality (b) we have introduced the complex line
  element $dz$ with the help of the tangential vector $i\,
  (z-z_0)/\rho$, substituting $(z-z_0)\, d\calH^1 = -i\rho\, dz$.  In
  equality (c) we have used the fact that the contour of integration
  can be deformed without changing the value of the integral; we
  deform the contour into increasingly large circles, and the limiting
  value is $0$. We exploited here $k+m \ge 1$.

  We have thus seen that one of the two contributions to $I_{m,k}$
  vanishes. Using again the tangential line element $(z-z_0)\,
  d\calH^1 = -i\rho\, dz$ we have
  \begin{align*}
    I_{m,k} &= \frac12 \int_{\del\Sigma} z^{-k}\, (z-z_0)^m\, d\calH^1
    = \frac{-i\rho}{2} \int_{\del\Sigma} z^{-k}\, (z-z_0)^{m-1}\, dz\,.
  \end{align*}
  At this point we have verified the claims for $m = 0$, since for
  $k\ge 1$ we can again move the contour of integration to $\infty$,
  hence the integral vanishes. In the case $m\ge 1$ we expand the term
  $(z-z_0)^{m-1}$ and find
  \begin{align*}
    I_{m,k} &= \frac{-i\rho}{2} \int_{\del\Sigma} z^{-k}\,
    \sum_{j=0}^{m-1} \vectdue{m-1}{j} z^{m-1-j} (-z_0)^j\ dz\,.
  \end{align*}
  By the residue theorem, the boundary integral is non-vanishing only
  for $j = m-k$, since for this value the exponent of $z$ is $m-1-j-k
  = -1$. We obtain
  \begin{align*}
    I_{m,k} &= \frac{-i\rho}{2} \int_{\del\Sigma} z^{-1}\,
     \vectdue{m-1}{m-k} (-z_0)^{m-k}\ dz
     = \pi\rho \vectdue{m-1}{m-k} (-z_0)^{m-k}
  \end{align*}
  by Cauchy's theorem. This proves the explicit formula \eqref
  {eq:Imk-explicit}.

  \smallskip For fixed $k\ge 1$, we expand the function
  $\Re(z^{-k})\lfloor \del B_\rho(z_0) : \del B_\rho(z_0) \to \R$ in
  spherical harmonics with coefficients $\gamma_l, \hat \gamma_l\in
  \R$,
  \begin{equation}
    \label {eq:expansion-newcenter-proof}
    \Re(z^{-k}) = \sum_{l\in \N} \gamma_l \Re (z-z_0)^l 
    + \hat\gamma_l \Im (z-z_0)^l\,.
  \end{equation}
  For arbitrary $m\ge 1$, we multiply this function with $\Re (z-
  z_0)^m$ and integrate over $\Gamma := \del B_\rho(z_0)$. Using
  orthogonality properties of spherical harmonics, we find
  \begin{align*}
    \int_\Gamma \Re(z^{-k}) \Re (z- z_0)^m\, d\calH^1
    = \sum_{l\in \N} \gamma_l \int_\Gamma \Re (z-z_0)^l \Re (z- z_0)^m\, d\calH^1
    = \gamma_m \pi \rho^{2m+1}\,.
  \end{align*}
  The left hand side is nothing else than $\Re(I_{m,k})$. Repeating
  the calculation for the imaginary part, we find $\gamma_m + i \hat
  \gamma_m = (\pi \rho^{2m+1})^{-1} I_{m,k}$. This verifies the
  expansion \eqref {eq:expansion-newcenter}.

  \smallskip Estimate \eqref {eq:Q-Imk-est} is obtained with a
  straightforward calculation. For $m\ge 1$ there holds
  \begin{align*}
    \sum_{k\in \N} Q^k |I_{m,k}| \le 
    \sum_{1\le k\le m} \pi \rho \vectdue{m-1}{m-k} |z_0|^{m-k} Q^k
    = \pi \rho Q\, (|z_0| + Q)^{m-1}\,.
  \end{align*}
  This completes the proof. 
\end{proof}

\subsection{Non-resonance for a non-concentric circular core}

The following result generalizes Proposition~\ref{prop:spher-incl} to
a geometry that is not radially symmetric.

\begin{theorem}[Non-resonance for non-concentric core]
  \label{thm:eccentric-localization}
  We consider a configuration of the following form. The coefficients
  are given by a circular core $\Sigma = B_\rho(z_0)$ with $0\in
  \Sigma\subset B_1(0)$ through \eqref {eq:a_eta}--\eqref
  {eq:A-def}. The source is located at a radius $q$ with $q > R^3$,
  and given as $f = F \calH^1\lfloor \del B_q(0)$ with $F = \sum_{k\ge
    1} \alpha_k \cos(k\theta) + \sum_{k\ge 1} \beta_k
  \sin(k\theta)$. We assume that the Fourier coefficients satisfy
  $\sum_{k} \left\{ |\alpha_k| + |\beta_k| \right\} \le C$.

  There exist $\eps_0 = \eps_0(q) > 0$ and $\eps_1 = \eps_1(q) > 0$
  such that, if the core is close to the unit disk in the sense that
  $|z_0| < \eps_0$ and $1-\eps_1 < \rho < 1$, the configuration is
  non-resonant.
\end{theorem}

\noindent\textit{Remark.} We will provide an explicit condition regarding the
smallness of $\eps_0$ and $\eps_1$, see \eqref {eq:choice-eps0eps1}
and \eqref {eq:eps0eps1-2}.

Note that we consider sources at the radius $q > R^3$. We know that
the source radius $q$ must satisfy $q > R^* = R^{3/2}$ to be
non-resonant, see Theorem \ref{thm:inner-resonance}. The lower bound
$R^3$ is probably not optimal.

\begin{proof} 
  As in the proof of Proposition \ref {prop:spher-incl}, we can
  decompose the expansion of $F$ into two parts and can write $F =
  F_\even + F_\odd$. By linearity of the equations it is sufficient to
  show the non-resonance property for the two contributions
  separately. Without restriction of generality, we study in the
  following $f = \sum_k \alpha_k f_k$ with $f_k = f_k^q$ of \eqref
  {eq:form-f-special}, and coefficients $(\alpha_k)_k \in l^2(\N;\R)
  \cap l^1(\N;\R) $.  

  We fix a sequence $\eta = \eta_j \searrow 0$.  Our aim is to
  construct a sequence $(v_\eta, w_\eta)$ of bounded energy $I_\eta$,
  and to use the primal variational principle \eqref{eq:primal} to
  show non-resonance.

  \medskip
  \noindent\emph{Step 1. Construction of comparison functions.} 
  We will use a construction similar in spirit to that used in the
  proof of Proposition \ref {prop:spher-incl}. The main difference is
  that the functions $\hat v_k(x)$ of \eqref {eq:hat-v-k-radial} are
  not suited for the eccentric core $\Sigma$. We will replace these
  functions by $\tilde v_k(x)$ of \eqref{eq:tilde-v-k-radial} defined
  in Step~1a. We then need to correct errors on $\del \Sigma$ due to
  the non-concentric geometry. This is done in Step~1b and Step~1c.

  \smallskip 
  \noindent\emph{Step 1a. Construction of the main part of $v_{\eta}$.}  
  We first construct the main part of $v_{\eta}$, denoted by
  $V_{\eta}$, following the construction of Proposition \ref
  {prop:spher-incl} with the new elementary functions
  \begin{equation}
    \label{eq:tilde-v-k-radial}
    \tilde v_k(x) :=
    \begin{cases}
      \tilde{v}_k(x) \vert_{\Sigma}\quad  &\text{ for } x \in \Sigma,\\
      r^{-k} \cos(k\theta)\quad &\text{ for } x\in B_R(0)\setminus \Sigma,\\
      r^k R^{-2k} \cos(k\theta)\quad &\text{ for } R < |x|\le q,\\
      r^{-k} (q/R)^{2k} \cos(k\theta)\quad &\text{ for } q <|x|,
    \end{cases}
  \end{equation}
  where $\tilde{v}_k \vert_{\Sigma}$ is chosen so that $\tilde{v}_k$
  is harmonic in $\Sigma$ and continuous on $\del\Sigma$.
  
  The function $V_\eta$ is constructed as a linear combination of the
  elementary functions $\tilde v_k$. The coefficients $\lambda_k$ are
  chosen to satisfy $\nabla\cdot (A\nabla V_\eta) = f$ away from
  $\del\Sigma$.  This leads to
  \begin{equation}
    \label{eq:v-eta-low-ecc}
    V_\eta := \sum_{k \in \N} \lambda_k\, \tilde v_k\,,\qquad
    \lambda_k := -\alpha_k \frac{q}{2k} q^{-k} R^{2k}\,.
  \end{equation}
  The coefficients $\lambda_k$ are actually identical to those in
  \eqref {eq:v-eta-low}. This is because $\tilde v_k$ coincides with
  $\hat v_k$ on $\R^2\setminus \Sigma$. 

  \smallskip 
  \noindent\emph{Step 1b.  Evaluation of errors on $\del\Sigma$.}  In
  the case of concentric spheres, the construction could be finished
  at this point, the distinction into high and low frequencies was
  only necessary in order to find the optimal bound for $q$.  Instead,
  since we now study a core $\Sigma$ that is not concentric, the
  functions $\tilde v_k$ are not solutions of $\nabla\cdot (A\nabla v)
  = 0$ on $\del\Sigma$. Hence, we need to correct the error on $\del
  \Sigma$:
  \begin{equation}
    \label{eq:F-error}
    \begin{aligned}
      F &:= \nabla \cdot (A \nabla V_\eta) - f 
      = \sum_{k \in \N} \lambda_k   
      \left\{ \del_\nu \tilde v_k|_\out + \del_\nu \tilde v_k|_\inn
      \right\} \, \calH^1\lfloor \del\Sigma \\
      & = \sum_{k \in \N} \lambda_k \del_\nu \tilde v_k|_\out \,
      \calH^1\lfloor \del\Sigma + \sum_{k \in \N} \lambda_k \del_\nu
      \tilde v_k|_\inn \, \calH^1\lfloor \del\Sigma \\
      & \equiv F_\out\, \calH^1\lfloor \del\Sigma + F_\inn\,
      \calH^1\lfloor \del\Sigma,
    \end{aligned}
  \end{equation}
  where the last equality gives the definition of $F_\out$ and
  $F_\inn$.

  We start with $F_\out$, the contributions from $\del_\nu \tilde
  v_k|_\out$, which can be calculated explicitly, since $\tilde v_k$
  is defined in \eqref {eq:tilde-v-k-radial} as $\tilde v_k(z) =
  r^{-k} \cos(k\theta) = \Re(z^{-k})$ on $x\in B_R(0)\setminus
  \Sigma$.  As in \eqref {eq:complex-normal-der} we calculate the
  normal derivative, which we then expand using
  \eqref{eq:expansion-newcenter},
  \begin{align*}
    \del_\nu \tilde v_k|_\out = \frac{-k}{\rho} \Re (z^{-k})
    =  \Re \sum_{m\ge k} \frac{-k}{\pi \rho^{2m+2}} \bar I_{m,k}\cdot
    (z- z_0)^m\,.
  \end{align*}
  Hence, we have evaluated the first part of the error in
  \eqref{eq:F-error} to be
  \begin{equation}
    \label{eq:F-out-proof}
    F_\out 
    = \sum_{k,m \ge 1} \lambda_k \frac{-k}{\pi \rho^{2m+2}}\
    \Re\left( \bar I_{m,k}\cdot (z- z_0)^m\right)
    = \sum_{m \ge 1} \Re\left( \mu_m^\out \cdot (z- z_0)^m\right)
  \end{equation}
  with 
  \begin{equation}\label{eq:mu-out}
    \mu_m^\out := \sum_{1\le k\le m} \lambda_k \frac{-k}{\pi \rho^{2m+2}}\,
    \bar I_{m,k}\,.
  \end{equation}
  We next estimate the decay of $|\mu_m^\out|$ as $m\to \infty$ with
  the help of estimate \eqref{eq:Q-Imk-est} for $|I_{m,k}|$. We set $Q
  := R^2/q < 1/R$ and use the sequence $\beta_k := Q^k
  |I_{m,k}|$. Using the elementary estimate
  \begin{equation*}
    \| (\beta_k)_k \|_{l^2}^2 \le \|(\beta_k)_k\|_{l^\infty} \|
    (\beta_k)_k \|_{l^1} \le \| (\beta_k)_k \|_{l^1}^2,
  \end{equation*}
  we obtain
  \begin{equation}\label{eq:mu-out-est}
    \begin{split}
      | \mu_m^\out | 
      &\le C \sum_{k\le m} |\alpha_k|\, (R^2/q)^k\rho^{-2m} | I_{m,k}| 
      \le C \| (\alpha_k)_k \|_{l^2}\ \rho^{-2m}\, \| (\beta_k)_k \|_{l^1}\\
      &\le C \| (\alpha_k)_k \|_{l^2}\ \rho^{-2m}\, 
      (|z_0| + (R^2/q))^{m}\,.      
    \end{split}
  \end{equation}
  To guarantee fast decay of $|\mu_m^\out|$, we choose 
  $\eps_0$ and $\eps_1$ such that 
  \begin{equation}
    \label{eq:choice-eps0eps1}
    \frac{1}{(1-\eps_1)^2} \left(\eps_0 + \frac{R^2}{q}\right) 
    < \frac1{R}.
  \end{equation}
  This is possible since $q > R^3$.  Combined with our assumptions
  $|z_0|<\eps_0$, and $1-\rho<\eps_1$, we have obtained the estimate
  \begin{equation} \label{eq:mu_m_out}
    \lvert \mu_m^\out \rvert \le C \| (\alpha_k)_k \|_{l^2}\ R^{-m}.
  \end{equation}

  We next study the other error contribution $F_{\inn}$ in
  \eqref{eq:F-error}. Our goal is to express, analogous to \eqref
  {eq:F-out-proof},
  \begin{equation}
    \label{eq:F-inn-proof}
    F_\inn = \sum_{k \in \N} \lambda_k \del_\nu \tilde v_k|_\inn 
    = \sum_{m \ge 1} \Re\left( \mu_m^\inn \cdot (z- z_0)^m\right)\,,
  \end{equation}
  and to provide an estimate for the coefficients $\mu_m^\inn$.

  As a first step we expand the function $\tilde v_k(z) = \Re(z^{-k})$
  on $\del\Sigma$, which was done in \eqref
  {eq:expansion-newcenter}. Since both components of the function
  $z\mapsto (z- z_0)^m$ are harmonic in $\Sigma$, the expansion of the
  boundary values provides us also with the harmonic extension $\tilde
  v_k \vert_{\Sigma}$. Formula \eqref {eq:expansion-newcenter} yields
  \begin{align*}
    \tilde v_k(z) = \Re \sum_{m\ge k} (\pi \rho^{2m+1})^{-1} \bar
    I_{m,k}\cdot (z- z_0)^m\quad \text{ for } z \in \Sigma\,.
  \end{align*}
  We next evaluate the normal derivative, using again
  \eqref{eq:complex-normal-der}.  We find, on $\del\Sigma$,
  \begin{align*}
    \del_\nu \tilde v_k|_\inn =  \sum_{m\ge k} (\pi \rho^{2m+1})^{-1} \Re (I_{m,k})
    \frac{m}{\rho} \Re (z- z_0)^{m}
    - \sum_{m\ge k} (\pi \rho^{2m+1})^{-1} \Im (I_{m,k})
    \frac{m}{\rho} \Im (z- z_0)^{m}\,.
  \end{align*}
  This provides for the normal derivative from inside $\Sigma$ the
  expansion \eqref {eq:F-inn-proof} with the coefficients
  \begin{equation}
     \mu_m^\inn = \sum_{k\le m} \lambda_k (\pi \rho^{2m+1})^{-1}
    I_{m,k} \frac{m}{\rho}\,.
  \end{equation}
  This expression for $\mu_m^\inn$ is analogous to \eqref {eq:mu-out}
  for $\mu_m^\out$. In particular, $\mu_m^\inn$ can be treated
  similarly to $\mu_m^\out$ in \eqref {eq:mu-out-est}.

  To sum up, we have obtained
  \begin{equation}
    \label{eq:F-error-expanded}
    F := \Re  \sum_{m = 1}^\infty 
    \mu_m\cdot (z-z_0)^m\ \calH^1\lfloor \del\Sigma\,,
    \quad \text{with }\
    | \mu_m | \le C \| (\alpha_k)_k \|_{l^2}\ R^{-m}\,. 
  \end{equation}

  \smallskip 
  \noindent\emph{Step 1c. Correcting the error on $\del\Sigma$.}  We
  now correct the error term $F$ given by \eqref
  {eq:F-error-expanded}. We recall that the error was introduced by
  $V_\eta$ through $\nabla\cdot (A\nabla V_\eta) = F$.

  We define $m^*$ to be the smallest integer with $(\rho/R)^{2m^*} \le
  \eta$ and decompose accordingly
  \begin{align*}
    F &= F^\low \ \calH^1\lfloor \del\Sigma+ F^\high\ \calH^1\lfloor \del\Sigma \\
    &\equiv
    \Re  \sum_{m \le m^*} \mu_m\cdot (z-z_0)^m\ \calH^1\lfloor \del\Sigma
    + \Re \sum_{m > m^*} \mu_m\cdot (z-z_0)^m\ \calH^1\lfloor \del\Sigma\,.
  \end{align*}
  We will correct the high frequency error $F^\high$ by taking
  $w_{\eta}$ to be the solution to $\Delta w_\eta = F^\high$.

  The low frequency error $F^\low$ must be treated with a quite
  different approach.  The basic idea is to use the perfect plasmon
  waves $\hat V_k$ and $\tilde V_k$ as in
  \eqref{eq:choice-psi-emptyset}. We define $\hat V_k(z) = \Re(z^k)$
  and $\tilde V_k(z) = \Im(z^k)$ for $z\in B_R(0)$, and $\hat V_k(z) =
  R^{2k} \Re(z^{-k})$ and $\tilde V_k(z) = R^{2k} \Im(z^{-k})$ for
  $z\in \R^2\setminus B_R(0)$.  These functions are perfect plasmon
  waves for the curve $\del B_R(0)$, but they are not solutions on
  $\del\Sigma$. In other words, the nonzero functions $\nabla\cdot
  (A\nabla \hat V_k)$ and $\nabla\cdot (A\nabla \tilde V_k)$ are
  concentrated on $\del\Sigma$.

  The normal derivatives on $\del\Sigma$ of these functions have been
  used before.  There holds $\nu\cdot \nabla \hat V_k = \rho^{-1} k
  \Re ((z-z_0)(z^{k-1}))$, compare \eqref
  {eq:complex-normal-der}. Similarly, we have for the imaginary part
  the normal derivative $\nu\cdot \nabla \tilde V_k =
  \Re(\overline{z-z_0}\cdot ik \rho^{-1} \bar z^{k-1}) = \rho^{-1} k \Im
  ((z-z_0)(z^{k-1}))$.

  The fact that the functions $\hat V_k$ and $\tilde V_k$ are
  \emph{not} solutions on $\del\Sigma$ can be used to our
  advantage. We expand the low frequency error $F^\low$ in terms of
  the residuals of these functions. Expanding with respect to the
  center $z = 0$, we find, on $\partial \Sigma$
  \begin{align*}
    F^\low &= \Re  \sum_{m = 1}^{m^*} \mu_m (z-z_0)^{m-1} (z-z_0)
    = \Re  \sum_{m = 1}^{m^*} \mu_m \sum_{k=1}^{m} 
     \vectdue{m-1}{k-1} z^{k-1} (-z_0)^{m-k} (z-z_0)\\
     &= \sum_{k = 1}^{m^*} \left( \Re\, \sum_{m = k}^{m^*} \mu_m 
       \vectdue{m-1}{k-1} (-z_0)^{m-k} \right) \Re\, (z^{k-1}(z-z_0))\\
     &\qquad - \sum_{k = 1}^{m^*} \left( \Im\, \sum_{m = k}^{m^*} \mu_m 
       \vectdue{m-1}{k-1} (-z_0)^{m-k} \right) \Im\, (z^{k-1}(z-z_0))\\
     &= \sum_{k = 1}^{m^*} \hat\beta_k\, \nu\cdot \nabla \hat V_k
     + \sum_{k = 1}^{m^*} \tilde\beta_k\, \nu\cdot \nabla \tilde V_k\,,
  \end{align*}
  with the real coefficients $\hat\beta_k$ and $\tilde \beta_k$ given by
  \begin{align*}
    \hat\beta_k - i \tilde \beta_k = \frac{\rho}{k} \sum_{m = k}^{m^*} 
    \mu_m \vectdue{m-1}{k-1} (-z_0)^{m-k}\,.
  \end{align*}
  We use now the estimate \eqref {eq:F-error-expanded}, $|\mu_m| \le C
  R^{-m}$, to estimate the complex coefficient $\beta_k = \hat\beta_k
  + i \tilde \beta_k$,
  \begin{align*}
    |\beta_k| \le C \sum_{m = k}^{m^*} |\mu_m| \vectdue{m-1}{k-1}
    |z_0|^{m-k} \le C \sum_{m = k}^{m^*} R^{-m} (1+|z_0|)^{m-1}\,.
  \end{align*}
  We choose $\eps_0$ such that 
  \begin{equation}
    \label{eq:eps0eps1-2}
    (1+\eps_0) \le R.
  \end{equation}
  Hence, as $(1+|z_0|)/R < 1$, we have 
  \begin{align}\label{eq:est-beta-k}
    |\beta_k| \leq C \sum_{m = k}^{m^*} R^{-m} (1+|z_0|)^{m-1} \le C
    R^{-k} (1+|z_0|)^{k-1}\,.
  \end{align}
  We can therefore compensate the low frequency errors introduced by
  $V_\eta$ of \eqref{eq:v-eta-low-ecc} using the functions $\hat V_k$
  and $\tilde V_k$. We recall that $\nabla\cdot A\nabla \hat V_k = -2
  \del_\nu \hat V_k\,\calH^1\lfloor \del\Sigma$ and set therefore
  \begin{equation}
    \label{eq:v-eta-finally}
    v_\eta := V_\eta + \frac12 \sum_{k = 1}^{m^*} \hat\beta_k \hat V_k
    + \frac12 \sum_{k = 1}^{m^*} \tilde\beta_k \tilde V_k\,.
  \end{equation}
  With this choice of $v_\eta$, we have $\nabla\cdot (A\nabla v_\eta)
  = F + f - F^\low = f + F^\high$.  If we choose $w_{\eta}$ as the
  solution of $\Delta w_\eta = F^\high$, we obtain
  \begin{equation*}
    \nabla\cdot (A\nabla v_\eta) - \Delta w_{\eta} =  f + F^\high - F^{\high} = f\,.
  \end{equation*}
  In particular, the constraint of \eqref{eq:primal} is satisfied.

  \medskip
  \noindent\emph{Step 2. Calculation of energies}.
  It remains to calculate the energy $I_{\eta}(v_{\eta}, w_{\eta})$.
  By the triangle inequality, the energy of $v_{\eta}$ is bounded if
  we can control the energy of each term on the right hand side of
  \eqref{eq:v-eta-finally}.

  Recall that $\tilde{v}_k$ defined in \eqref{eq:tilde-v-k-radial}
  agrees with $\hat{v}_k$ defined in \eqref{eq:hat-v-k-radial} on
  $\R^2 \setminus \Sigma$.  Accordingly, the energy contribution of
  $V_\eta$ is bounded by a similar argument as in the proof of
  Proposition \ref{prop:spher-incl}. But since we do not have
  orthogonality of the basis functions, we calculate here with the
  $l^1$-assumption on $\alpha_k$.  Using \eqref {eq:v-eta-low-ecc},
  the triangle inequality, and the fact that we are in the case $q >
  R^2$, we find
  \begin{equation}
    \label{eq:V-eta-energy-ecc}
    \begin{split}
      \left( \eta \int |\nabla V_\eta|^2\right)^{1/2}
      &= \sqrt{\eta} \|\nabla V_\eta\|_{L^2}
      \le \sqrt{\eta}  \sum_{k} |\lambda_k|\, 
      \left(\int |\nabla \tilde v_k|^2\right)^{1/2}\\
      &\le C \sqrt{\eta} \sum_{k} |\alpha_k|\, q^{-k} R^{2k}\,
      q^{k} R^{-2k} 
      = C \sqrt{\eta} \sum_{k} |\alpha_k|
      \le C \sqrt{\eta} \,.
    \end{split}
  \end{equation}
  
  The calculations for the energies related to the corrections
  involving $\hat V_k$ and $\tilde V_k$ are identical, we therefore
  treat here only the contribution of $\sum \hat\beta_k \hat
  V_k$. Exploiting orthogonality and the estimate
  \eqref{eq:est-beta-k} for $\beta_k$, we find
  \begin{align*}
    &\eta \int \left|\nabla \sum_{k = 1}^{m^*} \hat\beta_k \hat
      V_k\right|^2
    \le C \eta \sum_{k = 1}^{m^*} |\hat\beta_k|^2 \int |\nabla \hat V_k|^2\\
    &\qquad \le C \eta \sum_{k = 1}^{m^*} R^{-2k} (1+|z_0|)^{2k} k 
    R^{2k} \le C (\rho / R)^{2m^{\ast}} (m^*)^2 (1+|z_0|)^{2m^*} \,,
  \end{align*}
  where in the last inequality, we have used the fact that $\eta \leq
  C (\rho / R)^{2m^{\ast}}$ by our choice of $m^{\ast}$. By the
  assumption $\rho < 1$ and the choice of $\eps_0$ in
  \eqref{eq:eps0eps1-2}, the energy is bounded.
  
  \smallskip

  It remains to estimate the energy contribution of $w_{\eta}$, given
  by the solution to $\Delta w_{\eta} = F^{\high}$. Note that the
  squared norm of $F^\high$ can be estimated by \eqref
  {eq:F-error-expanded} as
  \begin{align*}
    \| F^\high \|_{H^{-1}}^2 \le C \sum_{m > m^*} |\mu_m|^2 \rho^{2m}
    \le C \| (\alpha_k)_k \|_{l^2}^2\, (\rho/R)^{2m^*} \le C \eta\,.
  \end{align*}
  By the properties of the solution operator $(-\Delta)^{-1}$ acting
  on functions in $H^{-1}(\R^2)$ with vanishing average, we conclude
  that the energy contribution $\eta^{-1} \int |\nabla w_\eta|^2 \le C
  \eta^{-1} \eta \le C$ is bounded.

  The proof of non-resonance is complete by inequality
  \eqref{eq:lem-primal-conseq}.
\end{proof}

Regarding the last proof we remark that the decomposition of $F$ into
high and low frequency parts was only necessary in order to obtain an
improved lower bound for $q$.

\bigskip 
\noindent {\bf Acknowledgment.} This work was performed while BS was
visiting the Courant Institute in New York, partially funded by the
Deutsche Forschungsgemeinschaft under contract SCHW 639/5-1. The kind
hospitality is gratefully acknowledged.  The research of RVK and MIW
was partially supported by the National Science Foundation through
grants DMS-0807347 and DMS-1008855 respectively.

\appendix

\section{Are there perfect plasmon waves in dimensions
 $n\ne2$?}\label{perfectplasmons?}

The {\it perfect plasmon waves} $\hat\psi_k$, defined in
\eqref{eq:choice-psi-emptyset}, play a central role in the
construction of trial functions in our variational arguments. In this
section we consider the question of whether such plasmons exist in
dimensions $n\ne2$. Plasmon solutions may be viewed as solutions to
the {\it plasmonic eigenvalue problem} (see \cite{Grieser:12} and
references cited therein). Specifically, for a fixed radius, $R>0$, we
seek $\psi(x),\ x\in\mathbb{R}^n$, such that
\begin{align}
  \nabla\cdot \left(A\nabla \psi\right) &=\ 0,\ \ \psi\to0\ \ {\rm
    as}\ \ r=|x|\to\infty
 \label{plasmonpde}
\end{align}
where
\begin{equation}
  A(x) = 
  \begin{cases}
    1\ , &  |x|>R;\\
    \eps\ , & |x|<R.\\
  \end{cases}
  \label{plasmonicthing}
\end{equation}
This is the situation where we take the core $\Sigma=\emptyset$.
\medskip

A function $\psi$ is a weak solution of
\eqref{plasmonpde}-\eqref{plasmonicthing} if and only if
\begin{align}
  & \Delta \psi\ =\ 0,\ |x|\ne R \label{harmonic}\\
  &\left.\psi\right|_{|x|=R^-}\ =\ \left.\psi\right|_{|x|=R^+}\ \
  \textrm{(continuity)} \label{continuity}\\
  & -\eps\ \left.\frac{\partial \psi}{\partial r}\right|_{|x|=R^-}\ 
  +\ \left.\frac{\partial \psi}{\partial r}\right|_{|x|=R^+}\ =\ 0\ 
  \ \textrm{(flux continuity)}\label{flux}\\
  &\psi\to0\ {\rm as}\ |x|\to\infty\ . \label{decay}
\end{align}

We are interested in the case $\eps=-1$ and we show

\medskip
\noindent {\it Claim 1:} There are no plasmons localized on spheres in
dimension $n=3$.

\medskip
\noindent {\it Claim 2:} There \underline{are} plasmons localized at
the plane separating half-spaces in \underline{any} spatial dimension
$n\ge2$. That is, for any $n\ge2$, there are solutions of
\eqref{harmonic}-\eqref{decay} in the case where the spherical
interface $|x|=R$ is replaced by the planar interface $x_n=0$.

\medskip
\noindent {\it Proof of Claim 1:}\ Let $Y^l(\Omega),\ \Omega\in S^2$
denote any spherical harmonic of order $l$. That is,
\begin{equation}
-\Delta_{S^2}Y^l(\Omega)\ =\ l(l+1)Y^l(\Omega),\ \ l\ge0.
\label{eig}
\end{equation}
The corresponding solutions of Laplace's equation in dimension $3$ are:
\begin{equation}
r^l\ Y^l(\Omega)\ \ {\rm and}\ \ r^{-l-1}\ Y^l(\Omega),\ \ r=|x|.
\label{harmonics}
\end{equation}
Regularity away from the interface and decay at infinity imply
\begin{equation*}
 \psi(x) = 
 \begin{cases}
   r^l \ Y^l(\Omega)\ , &  0\le r\le R;\\
   c\ r^{-l-1} \ Y^l(\Omega)\ , & r> R.\\
 \end{cases}
\end{equation*}
The constant $c$ is determined by the interface conditions.
Continuity implies $c=R^{2l+1}$. Furthermore, the left hand side of
\eqref{flux} is equal to $-R^{l-1}\ne0$. Hence there are no perfect
spherical plasmon waves in dimension $3$.  A similar calculation holds
in all dimensions $n\ge 4$, where
$-\Delta_{S^{n-1}}Y^l(\Omega)=l(l+n-2)Y^l(\Omega)$.

\begin{remark} Note that
 \begin{equation*}
   \psi_l(x) = \left\{ 
     \begin{array}{cl} r^l \ Y^l(\Omega)\ , &  0\le r\le R\\
       R^{2l+1}\ r^{-l-1} \ Y^l(\Omega)\ , & r>R\\
     \end{array} \right.  \ .
 \end{equation*} 
 is a plasmonic eigenstate with corresponding eigenvalue
 $\eps_l=-\frac{l}{l+1}$. This sequence of plasmonic eigenvalues
 approaches $-1$ as $l\to\infty$.
\end{remark}

\noindent {\it Proof of Claim 2:}\ For any $n\ge2$, we write
$x\in\mathbb{R}^n$ as $x=(x_\perp,x_n)$ and define:
\begin{equation*}
 \psi(x;\xi) = \left\{ 
   \begin{array}{cl} e^{-|\xi_n|x_n}\ e^{i\xi_\perp\cdot x_\perp} , &  x_n>0\\
     e^{|\xi_n|x_n}\ \ e^{i\xi_\perp\cdot x_\perp} , & x_n<0\\
      \end{array} \right.  
\end{equation*}
where $\xi_n\in \mathbb{R}$ and $\xi_\perp\in\mathbb{R}^{n-1}$ are chosen so that:
\[ |\xi_n|^2\ =\ \xi_\perp\cdot \xi_\perp\ .\] Then $\psi(x;\xi)$ is a
plasmon wave.

\bibliography{lit_cloaking}

\end{document}